\documentclass[10pt]{amsart}
\usepackage{amsxtra, amsfonts, amsmath, amsthm, amstext, amssymb, amscd, mathrsfs, verbatim, color}
\usepackage{threeparttable}
\usepackage{mathtools}
\usepackage[ansinew]{inputenc}\usepackage[T1]{fontenc}
\usepackage[all,cmtip]{xy}
\usepackage{ytableau}
\usepackage{young}
\theoremstyle{plain}
\usepackage{amsthm,amssymb}

\newtheorem{theorem}{Theorem}[section]
\newtheorem{lemma}[theorem]{Lemma}
\newtheorem{definition-theorem}[theorem]{Definition-Theorem}
\newtheorem{proposition}[theorem]{Proposition}
\newtheorem{corollary}[theorem]{Corollary}

\theoremstyle{definition}
\newtheorem{definition}[theorem]{Definition}
\newtheorem{example}[theorem]{Example}
\newtheorem{remark}[theorem]{Remark}
\newtheorem{notation}[theorem]{Notation}

\newcommand \bth[1] { \begin{theorem}\label{t#1} }
\newcommand \ble[1] { \begin{lemma}\label{l#1} }

\newcommand \bpr[1] { \begin{proposition}\label{p#1} }
\newcommand \bco[1] { \begin{corollary}\label{c#1} }
\newcommand \bde[1] { \begin{definition}\label{d#1}\rm }
\newcommand \bex[1] { \begin{example}\label{e#1}\rm }
\newcommand \bre[1] { \begin{remark}\label{r#1}\rm }

\newcommand \bnota[1] {\begin{notation}\label{n#1}\rm }
\newcommand {\ele} { \end{lemma} }

\newcommand {\epr} { \end{proposition} }
\newcommand {\eco} { \end{corollary} }
\newcommand {\ede} { \end{definition} }
\newcommand {\eex} { \end{example} }
\newcommand {\ere} { \end{remark} }
\newcommand {\enota} { \end{notation} }

\begin{document}

\setlength{\baselineskip}{1.2\baselineskip}
\title[Dirac cohomology of standard modules]{A vanishing theorem for Dirac cohomology of standard modules}

\author[Kei Yuen Chan]{Kei Yuen Chan}
\address{ Department of Mathematics, University of Georgia, Athens GA}
\email{KeiYuen.Chan@uga.edu}

\begin{abstract}
This paper studies the Dirac cohomology of standard modules in the setting of graded Hecke algebras with geometric parameters. We prove that the Dirac cohomology of a standard module vanishes if and only if the module is not twisted-elliptic tempered. The proof makes use of two deep results. One is some structural information from the generalized Springer correspondence obtained by S. Kato and Lusztig. Another one is a computation of the Dirac cohomology of tempered modules by Barbasch-Ciubotaru-Trapa and Ciubotaru.

We apply our result to compute the Dirac cohomology of ladder representations for type $A_n$. For each of such representations with non-zero Dirac cohomology, we associate to a canonical Weyl group representation. We use the Dirac cohomology to conclude that such representations appear with multiplicity one.
\end{abstract}

\maketitle 
\section{Introduction}

\subsection{}
The Dirac operator plays an important role in the representation theory for real reductive groups. It was used for geometric constructions of discrete series by the work of Parthasarathy \cite{Pa} and Atiyah-Schmid \cite{AtSc}. The notion of Dirac cohomology was introduced by Vogan \cite{Vo} in around 1997 along with a conjecture relating to the infinitesimal characters of Harish-Chandra modules. The conjecture was later proved by Huang-Pand\v{z}i\'c \cite{HP}, and has been extended to many other settings e.g. \cite{HP2}, \cite{Ko}, \cite{BCT}, \cite{Ci2}, \cite{Ch2}. The problem of determining the Dirac cohomology of interesting modules has been considered in \cite{HP1, HP2, HX, DH, BPT}  and many others, and also has applications to branching rules \cite{HPZ}, and harmonic analysis and endoscopy theory \cite{Hu}.
 
Graded Hecke algebras were introduced by Lusztig \cite{Lu1, Lu89} for the study of $p$-adic reductive groups. Motivated by some analogies between $p$-adic and real reductive groups, Barbasch-Ciubotaru-Trapa \cite{BCT} define a Dirac operator and develop the Dirac cohomology theory for graded Hecke algebras. The theory has been further developed in \cite{CT, CM, COT} and there are applications on $W$-character formulas \cite{CT, CH} and a classification of discrete series \cite{CO}.

This paper studies the Dirac cohomology of standard modules for graded Hecke algebras. Determining their Dirac cohomology requires understanding on some fine structure. It turns out that part of the information can be obtained from results of generalized Springer correspondence by Kato \cite{Ka} and Lusztig \cite{Lu0,Lu1,Lu2,Lu3}  via the geometric realization of standard modules. More details and applications of our result will be discussed in next two sections.

\subsection{Main results}
For any complex reductive group $H$, let $H^{\circ}$ be the identity component of $H$.

We first introduce some notations. Let $G$ be a complex connected reductive group with its Lie algebra $\mathfrak{g}$. Let $\mathcal N_G$ be the set of nilpotent elements in $\mathfrak{g}$. Let $L$ be a Levi subgroup of $G$ and let $\mathcal O$ be a $L$-orbit in $\mathcal N_G$. Let $\mathcal L$ be a cuspidal local system of $\mathcal O$. The datum $(L, \mathcal O, \mathcal L)$ forms a cuspidal triple. Let $P$ be a parabolic subgroup with the Levi decomposition $LU$. Let $T=Z_L^{\circ}$, where $Z_L$ is the center of $L$. Let $\mathfrak{h}$ be the Lie algebra of $T$.  Let $\mathfrak{h}^{\vee}$ be the dual space of $\mathfrak{h}$. All these data determine a Weyl group $W$, a root system $R$, a set $\Pi$ of simple roots in $R$ and a parameter function $c:\Pi\rightarrow \mathbb R$ (see Section \ref{ss geo gha} for the detailed notations). Lusztig \cite{Lu1, Lu2} constructed geometrically a graded Hecke algebra $\mathbb{H}$, which also has an algebraic description $\mathbb{H}(V,W,\Pi,\mathbf r, c)$ (see Section \ref{def gaha}).


Before defining the Dirac cohomology, we introduce more notations. Let $V'=\mathbb{C}\otimes_{\mathbb{Z}}R$. Let $C(V')$ be the Clifford algebra for $V'$ (Section \ref{ss dirac op}). Let $S$ be a fixed choice of simple module of $C(V')$. Let $\widetilde{W}$ be the spin double cover of $W$ (Section \ref{ss spin cover}). $\widetilde{W}$ then defines a twisted group algebra $\mathbb{C}[\widetilde{W}]$, which is a natural subalgebra of $C(V')$. There is also a diagonal embedding, denoted $\Delta: \mathbb{C}[\widetilde{W}] \rightarrow \mathbb{H} \otimes C(V')$, which plays an important role in the Dirac cohomology.

The Dirac element, denoted $D$, for $\mathbb{H}$ is an element in $\mathbb{H} \otimes C(V')$ which has a remarkable formula for $D^2$ (\ref{eqn formul d2}). Given an $\mathbb{H}$-module $X$, define the Dirac cohomology of $X$ (as in \cite{BCT}):
\[  H_D(X)=\frac{\mathrm{ker}~ \pi(D)}{\mathrm{ker} ~\pi(D)\cap \mathrm{im}~\pi(D)},
\]
where $\pi: \mathbb{H} \otimes C(V') \rightarrow \mathrm{End}_{\mathbb{C}}(X \otimes S)$ is the map defining the $\mathbb{H} \otimes C(V')$-action on $X \otimes S$. We also have a natural action of $\widetilde{W}$ on $X \otimes S$ via the diagonal embedding $\Delta$, and such action turns out to commute with the action of $D$, up to a sign. Hence, $H_D(X)$ is equipped with a $\widetilde{W}$-representation structure. A crucial result in the Dirac cohomology theory, so-called Vogan's conjecture \cite[Theorem 4.4]{BCT}, states that if $X$ is irreducible and $H_D(X)\neq 0$, then the $\widetilde{W}$-isotypic component in $H_D(X)$ determines the central character of $X$.

A goal of this paper is to compute the Dirac cohomology of standard modules. Let $e \in \mathcal N_G$ and let $r \in \mathbb{C}^{\times}$. Let $s \in \mathfrak{g}$ be a semisimple element satisfying $[s,e]=2re$.  Let $A(e,s)$ be the component group for $e$ and $s$, and let $\zeta \in \mathrm{Irr}A(e,s)$. For the datum $(s,e,r,\zeta)$, Lusztig geometrically constructs a module $E_{s,e,r,\zeta}$ (see Section \ref{ss deform struct rtm}). We shall assume that $E_{s,e,r,\zeta}$ is non-zero and call it a {\it standard module}. An interesting case is that when $s=h_e$ is the middle element in the $\mathfrak{sl}_2$-triple for $e$. In that case, $E_{rh_e,e,r,\zeta}$ is a tempered module (see Definition \ref{def temp mod}). When $e$ has a solvable centralizer in $\mathfrak{g}$, it is shown in \cite{Ci} and \cite{BCT} that $H_D(E_{rh_e,e,r,\zeta}) \neq 0$. We call $E_{rh_e,e,r,\zeta}$ to be a {\it twisted-elliptic tempered module}, motivated by the close connections to twisted elliptic pairings for Weyl groups and component groups \cite{CH} (also see \cite{Ch3}).

The first main result in this paper is the following:

\begin{theorem} \label{thm intro standard mod} (Theorem \ref{thm vanihsing dirac coh})
Let $E$ be a standard module as above. Then $H_D(E)= 0$ if and only if $E$ is not a twisted-elliptic tempered module.
\end{theorem}

We shall explain the strategy for proving Theorem \ref{thm intro standard mod} in the next subsection. We remark that in order to apply results of Kato and Lusztig, a main ingredient of our proof is to study certain deformation behavior of the Dirac cohomology, which is carried out in Section \ref{s def dirac coh}.

For unequal (not necessarily geometric) parameter case of type BC, we also deduce a version of Theorem \ref{thm intro standard mod} from results and methods of \cite{Ka,CK} in Section \ref{s unequal bc}.

We also obtain the twisted Dirac index of a standard module consequently. Here the twisted Dirac index of an $\mathbb{H}$-module $X$ (with a $\mathbb{Z}_2$-graded decomposition $X=X^+\oplus X^-$, see Section \ref{ss z2 grading}) is defined as:
\[  I(X) =X^+\otimes \mathcal S -X^-\otimes \mathcal S
\]
as virtual $\widetilde{W}$-representations. Here $\mathcal S=S$ when $\mathrm{dim}V'$ is even and $\mathcal S$ is the direct sum of two simple $C(V')$-modules when $\mathrm{dim}V'$ is odd. The terminology of the Dirac index is suggested by its relation to the Dirac cohomology in Section \ref{ss dirac index} (also see \cite{CT, CH, BPT}).
\begin{corollary} (Corollary \ref{cor dirac index})
Let $E$ be a standard module. Define $E'$ as in Lemma \ref{lem top generate}. Then $I(E')=0$ if and only if $E$ is not twisted-elliptic tempered.
\end{corollary}

Dirac index can descend to the Grothendieck group of $\mathbb{H}$-modules, but Dirac cohomology cannot. Hence Dirac cohomology is an invariant encoding more information. When $E$ is tempered, the Dirac index $I(E)$ is computed in \cite{CH}. A twisted Dirac index of standard modules for real groups has been recently computed in \cite{BPT}.



The Dirac index of a simple module can be, in principle, computed by using a character formula from a Kazhdan-Lusztig type polynomial. Perhaps a significance of Theorem \ref{thm intro standard mod} is that the Dirac cohomology of a simple module may be computed if a suitable categorification of the character formula is established. For example, such strategy can be done for the ladder representations, where a BGG-type resolution exists.

Ladder representations of $GL(n,F)$ are introduced by Lapid-Mingue\'z \cite{LM} for the study of a determinantal formula of Tadi\'c, and they coincide with the calibrated representations of Ram \cite{Ra} in the graded Hecke algebra level \cite{BC2}. Speh representations form a subclass of ladder representation. Hence our result for the Dirac cohomology of ladder representations can be viewed as an extension of the Speh representation case by Barbasch-Ciubotaru \cite{BC}. Interestingly, our method does not need a lot of knowledge on the $W$-structure of a ladder representation and hence we obtain some closed $W$-structure as a consequence.


\begin{theorem} \label{thm introd dirac coh ladd} (Theorems \ref{thm dirac ladder})
 Let $\mathbb{H}_l$ be the graded Hecke algebra of type $A_{l-1}$ (see Section \ref{ss zel and ladd} for notations). Let $X$ be a ladder representation whose central character is the same as that of a twisted-elliptic tempered module $E$. Then
\[  H_D(X) \cong H_D(E) \neq 0 .
\]
\end{theorem}

Theorem \ref{thm introd dirac coh ladd} leads to the question of identifying the $W$-representation which contributes the non-zero Dirac cohomology. We do that via the Arakawa-Suzuki functor. Such canonical $W$-type in terms of Young diagram parametrization can be constructed by a simple procedure starting from a Zelevinsky segment (Section \ref{ss partition alpha}).  Using the Dirac cohomology, we show that that $W$-type appears with multiplicity one:

\begin{corollary} (Theorem \ref{thm canonical unique})
Let $X$ be as in Theorem \ref{thm introd dirac coh ladd}. Let $W=S_l$ be the permutation group on $l$ elements. Then there exists a unique irreducible $W$-representation $\sigma$ such that
\[  \mathrm{Hom}_{\widetilde{W}}(H_D(X), \sigma \otimes S) \neq 0
\]
and
\[  \mathrm{Hom}_W(\sigma, X|_W) \cong \mathbb{C} .
\]
\end{corollary}

There are some other possible approaches for obtaining those $W$-structures. Our result for Dirac cohomology provides a feasible and simple argument for the proof.

\subsection{Outline of 'if' proof of Theorem \ref{thm intro standard mod}}

 Let $E$ be a standard module of $\mathbb{H}$ which is not tempered. The graded Hecke algebra $\mathbb{H}$ admits a natural $\mathbb{Z}_2$-grading (Section \ref{ss z2 grading}). For the introduction purpose, we assume that $E$ can be extended to a $\mathbb{Z}_2$-graded $\mathbb H$-module. 

Let $\mathbb{A}_W=S(V) \rtimes W$ be the skew group ring (Section \ref{ss operator aw}). We obtain an $\mathbb{A}_W$-module from $E$ as follows. Pick a remarkable $W$-subspace $\sigma$ of $E$ which comes from a generalized Springer correspondence and generates $E$. Define
\[  E_i = \left\{   px : x \in \sigma, \quad p \in S^{\leq i}(V) \right\},
\]
where $S^{\leq i}(V)$ be the space containing polynomials of degree less than or equal to $i$. Now set $\overline{E}_{\sigma}:= \bigoplus_i E_{i+1}/E_i$, which is equipped an $\mathbb{A}_W$-action with elements in $v \in V \subset S(V)$ sending $E_{i+1}/E_i \rightarrow E_{i+2}/E_{i+1}$ and elements in $w \in W $ sending $E_{i+1}/E_i \rightarrow E_{i+1}/E_i$. (The action coincides with the one in Section \ref{eqn iso graded algebra}.) Regarding $S(V)$ as a natural $\mathbb{Z}$-graded algebra, we also have a $\mathbb{Z}$-grading on $\mathbb{A}_W$. The module $\overline{E}_{\sigma}$ can be extended to a $\mathbb{Z}$-graded $\mathbb{A}_W$-module. 

To realize the $\mathbb{A}_W$-structure of $\overline{E}_{\sigma}$, we need to introduce a module of Kato. We define $\leq$ to be a partial ordering on $\mathrm{Irr}W$ from the closure ordering in the generalized Springer correspondence. For $\tau \in \mathrm{Irr}W$, we can associate to an $\mathbb{A}_W$-module $\mathbf K_{\tau}$ defined in \cite{Ka} (see (\ref{eqn kato mod}) for the precise definition).

\begin{lemma} \label{lem intro filtration on std}
There exists an $\mathbb{A}_W$-module filtration $\left\{ \overline{Y}_i \right\}_{i=0}^k$ on $\overline{E}_{\sigma}$ such that 
\[ 0=\overline{Y}_0 \subset \overline{Y}_1 \subset \overline{Y}_2 \subset \ldots \subset \overline{Y}_k=\overline{E}_{\sigma} \]
and for $i=0,\ldots, k-1$, $\overline{Y}_{i+1}/\overline{Y}_i$ is isomorphic to a module $\mathbf{K}_{\sigma_i}$ for some $\sigma_i \in \mathrm{Irr} W$. 
\end{lemma}

One can view $\mathbb{A}_W$ as a specific case of the general graded Hecke algebras and hence we have an analogue of the Dirac cohomology, which is denoted by $H_{\mathbb{D}_{\mathbb{A}}}$ (see Section \ref{ss operator aw}). We want to compute $H_{D_{\mathbb{A}}}(\mathbf K_{\sigma})$ ($\sigma \in \mathrm{Irr}W$).

The generalized Springer correspondence of Lusztig \cite{Lu0} and results of Kato \cite{Ka} show that for each $\tau \in \mathrm{Irr}W$, there is a tempered module, denote $X^{\tau}$ such that its associated $\mathbb{A}_W$-structure $\overline{X}^{\tau}$, satisfying
\[ \overline{X}^{\tau} \cong \mathbf K_{\tau} \]
(see Corollary \ref{cor struct h} for the precise statement).

On the other hand, the unitarity on $X^{\tau}$ implies that 
\[  H_D(X^{\tau}) \cong H_{D_{\mathbb{A}}}(\overline{X}^{\tau}) \cong H_{D_{\mathbb{A}}}(\mathbf K_{\tau}) \]
(see Proposition \ref{prop unitary dirac}). The Dirac cohomology $H_D(X^{\tau})$ has been known from the computations in \cite{Ci} and \cite{BCT}, and Vogan's conjecture \cite{BCT}. Hence $H_{D_{\mathbb A}}(\mathbf K_{\tau})$ is (almost) determined. 

We now turn back to the standard module $E$. To transfer the information from $ H_{D_{\mathbb{A}}}(\mathbf K_{\tau})$ to $H_D(E)$, we need the following result, which is a version of Theorem \ref{cor dirac cohom def1}:
\begin{theorem}
$H_D(E)$ is a quotient of $\bigoplus_i H_{D_{\mathbb{A}}}(\mathbf K_{\sigma_i})$.
\end{theorem}
Now with Lemma \ref{lem intro filtration on std} and the computation of $H_{D_{\mathbb A}}(\mathbf K_{\tau})$, it suffices to see that the $\widetilde{W}$-representations in $H_{D_{\mathbb A}}(\mathbf K_{\sigma_i})$ cannot appear in $H_{D}(E)$. To this end, we show that the operator $D^2$ acts by a non-zero scalar on any corresponding isotypic components of $\widetilde{W}$ from $H_{D_{\mathbb A}}(\mathbf K_{\sigma_i})$ in $E$. This proves $H_D(E)=0$. 

\subsection{Acknowledgments}
The problem of computing the Dirac cohomology of standard modules arose from discussions with Dan Ciubotaru. The author is grateful to him for several useful discussions and suggestions during this work. He also thanks Maarten Solleveld for useful comments for a earlier version of this paper. The author was supported by the Croucher Postdoctoral Fellowship. 

\section{Preliminaries}

\subsection{Graded Hecke algebra $\mathbb{H}$} \label{ss gha}
In this section, we work over a more general setting. We shall specialize to the graded Hecke algebra of geometric parameters in Section \ref{s struct temp mod}. 

Let $V$ be complex linear space and let $V^{\vee}=\mathrm{Hom}_{\mathbb{C}}(V,\mathbb{C})$. Let $W \subset O(V)$ be a group generated by simple reflections. Let $R$ be a (not necessarily reduced) root system associated to $W$ and let $\Pi$ be a fixed set of simple roots. Let $l:W \rightarrow \mathbb{Z}$ be the length function for $W$. For $\alpha \in R$, let $s_{\alpha}$ be the associated simple reflection for $V$. For $\alpha \in \Pi$, let $\alpha^{\vee}$ be the associated coroot i.e. $s_{\alpha}(v)=v-\alpha^{\vee}(v)\alpha$ for any $v \in V$. Let $w_0$ be the longest element in $W$.

Recall that $V^{\vee} \cong V$ as $W$-representations. This defines a symmetric bilinear form on $V$ given by $\langle v_1,v_2 \rangle_{V}=v_1(\eta(v_2))$, where $\eta: V \rightarrow V^{\vee}$ is an isomorphism. We shall normalize $\langle,\rangle_{V}$ such that  $\langle \alpha,\alpha \rangle_V=2$ for $\alpha \in R$. We shall sometimes write $\langle , \rangle$ for $\langle ,\rangle_V$.

Let $c: \Pi \rightarrow \mathbb{R}$ be a $W$-invariant parameter function. Write $c_{\alpha}$ for $c(\alpha)$. Let $S(V)$ be the polynomial algebra for $V$. Let $\mathbb{C}[W]$ be the group algebra of $W$.

\begin{definition} \cite{Lu1} \label{def gaha}
The graded Hecke algebra $\mathbb{H}=\mathbb{H}(V,W,\Pi, \mathbf r, c)$ is a complex associative algebra with generators $\left\{ p \in S(V) \right\}$, $\left\{ t_w : w \in W \right\}$, $\mathbf{r}$ satisfying the relations:
\begin{enumerate}
\item the natural map $S(V)$ to $\mathbb{H}$ given by $p \mapsto v$ is an algebra injection;
\item the natural map $\mathbb{C}[W]=\oplus_{w\in W} \mathbb{C}f_w$ to $\mathbb{C}[W]$ given by $f_w \mapsto t_w$ is an algebra injection;
\item for $\alpha \in R$, $t_{s_{\alpha}}v-s_{\alpha}(v)t_{s_{\alpha}}=\mathbf{r}c_{\alpha}\alpha^{\vee}(v)$;
\item $\mathbf{r}$ is in the center of $\mathbb{H}$.
\end{enumerate}
In particular, $\mathbb{H} \cong S(V)\otimes \mathbb{C}[W]\otimes \mathbb{C}[\mathbf r]$ naturally as vectors spaces. Let $\theta: \mathbb{H} \rightarrow \mathbb{H}$ be an algebra automorphism determined by $\theta(v)=-w_{0}(v)$ and $\theta(t_w)=t_{w_0ww_0^{-1}}$. For an $\mathbb{H}$-module $X$, let $\theta(X)$ be an $\mathbb{H}$-module determined by
\[ \pi_{\theta(X)}(h) x = \pi_X(\theta(h)) x ,\]
where $\pi_X$ (resp. $\pi_{\theta(X)}$) is the map defining the action of $\mathbb{H}$ on $X$ (resp. $\theta(X)$).
\end{definition}

\subsection{$\mathbb{Z}_2$-grading on $\mathbb{H}$-modules} \label{ss z2 grading}
For $v \in V \subset S(V)$, define $\widetilde{v}=\frac{1}{2}(v-v^*)$. Define a Hermitian-linear anti-involution $*: \mathbb{H}\rightarrow \mathbb{H}$ determined by
\[   \widetilde{v}^{*}=-\widetilde{v} \quad (v \in V), \quad t_w^*=t_{w^{-1}} \quad (w \in W), \quad \mathbf r^*=\mathbf r
\]
An $\mathbb{H}$-module $X$ is said to be {\it unitary} if there exists a non-degenerate positive-definite Hermitian form $\langle .,.\rangle^*$ on $X$ satisfying the property
\[ \langle h.x_1,x_2 \rangle^* = \langle x_1, h^*.x_2 \rangle^*.
\]
The terminology of unitarity comes from $p$-adic groups \cite{BM}. 

 The algebra $\mathbb{H}$ admits a natural $\mathbb{Z}_2$-grading given by $\mathrm{deg}(\widetilde{v})=1$ (mod $2$) and $\mathrm{deg}(t_w)=0$ (mod $2$). It is straightforward to check from the defining relations that the $\mathbb{Z}_2$-grading is well-defined. An $\mathbb{H}$-module $X$ is said to be $\mathbb{Z}_2$-graded if there exists a decomposition of $X=X^+ \oplus X^-$ such that $\widetilde{v}.X^{\pm}=X^{\mp}$ for any $v \in V$. 

We state two useful lemmas:

\begin{lemma} \label{lem z2 grading}
 Let $X$ be an $\mathbb{H}$-module. If $X \cong \theta(X)$, then $X$ can extend to a $\mathbb{Z}_2$-graded module. Moreover, if $X$ is irreducible, then the $\mathbb{Z}_2$-grading is unique.
\end{lemma}

\begin{proof}
Note that $t_{w_0}\theta$ has order $2$. Let $X^{\pm}$ be the eigenspace of the eigenvalue $\pm 1$ for $t_{w_0}\theta$. By using $t_{w_0}\theta(\widetilde{v})=-\widetilde{v}$ for any $v \in V$, one shows that the decomposition $X=X^+ \oplus X^-$ gives the desired $\mathbb{Z}_2$-grading. 
\end{proof}

\subsection{The Dirac operator $D$} \label{ss dirac op}
Let $V'=\mathbb{C}\otimes_{\mathbb Z} R\subset V$. Let $C(V')$ be the Clifford algebra associated to $\mathfrak{h}^{\vee}$ over $\mathbb{C}$ given by $V' \otimes V' / \langle v \otimes v+ \langle v, v\rangle : v \in V'\rangle $.

Let $n=\mathrm{dim} V'$. The space $V'$ admits a natural real form, which is $V_0':=\mathbb{R}\otimes_{\mathbb Z} R$. Let $\epsilon_1,\ldots, \epsilon_n$ be an orthogonal basis for $V_0'$ with respect to $\langle , \rangle$. There is a natural embedding $V' \hookrightarrow C(V')$, which will be denoted by $v \mapsto \widetilde{g}_v$. For $\alpha \in R^+$, define 
\[ \widetilde{s}_{\alpha}=\frac{\widetilde{g}_{\alpha}}{|\alpha|} .\]

 Following \cite{BCT}, define the Dirac element $D \in \mathbb{H} \otimes C(V')$
\[ D=\sum \widetilde{\epsilon_i} \otimes \widetilde{g}_{\epsilon_i} . \]
The element $D$ is independent of a choice of an orthogonal basis. The formula for $D^2$ is given by \cite[Theorem 3.5]{BCT}
\begin{align} \label{eqn formul d2}
   D^2 =-\sum_{i=1}^n \epsilon_i^2 \otimes 1 -\frac{1}{4}\mathbf r^2\sum_{\alpha>0, \beta>0, s_{\alpha}(\beta)<0} c_{\alpha}c_{\beta}|\alpha||\beta|s_{\alpha}s_{\beta}\otimes \widetilde{s}_{\alpha}\widetilde{s}_{\beta} 
\end{align}

Let $S$ be a fixed choice of simple module of $C(V')$. (When $\mathrm{dim}V'$ is odd, $C(V')$ has two isomorphic classes of simple modules. When $\mathrm{dim}V'$ is even, $C(V')$ has only one isomorphism class of simple modules.) By the definition of $D$, we have
\[  D(X^+ \otimes S) \subset X^- \otimes S, \quad D(X^+\otimes S) \subset X^- \otimes S .\]


The Dirac cohomology of an $\mathbb{H}$-module $X$ is defined as:
\begin{eqnarray} \label{eqn dirac coh}
   H_D(X) =\frac{\mathrm{ker}~ \pi(D)}{\mathrm{ker}~ \pi(D) \cap \mathrm{im}~ \pi(D)} ,
\end{eqnarray}
where $\pi: \mathbb{H} \otimes C(V') \rightarrow \mathrm{End}_{\mathbb{C}}(X \otimes S)$ is the map defining the $\mathbb{H} \otimes C(V')$-action on $X \otimes S$. Later on, we shall simply write $\mathrm{ker} D$, $\mathrm{im} D$ for $\mathrm{ker}\pi(D)$ and $\mathrm{im}~\pi(D)$ respectively.

\subsection{The spin cover $\widetilde{W}$ } \label{ss spin cover}
For more details of the spin cover, we refer the reader to \cite{Ci} and \cite{BCT}. 

Let $\widetilde{W}$ be the group generated by the elements $\widetilde{g}_{\alpha}$ ($\alpha \in R^+$) in $C(V')$. The map determined by $\widetilde{g}_{\alpha} \mapsto s_{\alpha}$ is a two-to-one map. We call $\widetilde{W}$ to be the spin cover of $W$. Let $\mathbb{C}[\widetilde{W}]$ be the subalgebra generated by  the elements $\widetilde{g}_{\alpha}$ ($\alpha \in R^+$). We define a diagonal embedding $\Delta: \mathbb{C}[\widetilde{W}] \rightarrow \mathbb{H} \otimes C(V')$ given by
\begin{align} \label{eqn diagonal embed}
  \widetilde{s}_{\alpha} \mapsto t_{s_{\alpha}} \otimes \widetilde{s}_{\alpha} .
\end{align}

\begin{lemma} \cite{Ci} (also see \cite[Lemma 4.1]{Ch1}) \label{lem center of tilde w}
The element 
\[\sum_{\alpha>0, \beta>0, s_{\alpha}(\beta)<0} c_{\alpha}c_{\beta}|\alpha||\beta|s_{\alpha}s_{\beta}\otimes \widetilde{s}_{\alpha}\widetilde{s}_{\beta} 
\]
lies in the center of $\Delta(\mathbb{C}[\widetilde{W}])$.
\end{lemma}

\begin{lemma} \cite[Lemma 3.4]{BCT} \label{lem commu rel D}
For any $\alpha \in R^+$, $\Delta(\widetilde{s}_{\alpha})D=-D\Delta(\widetilde{s}_{\alpha})$.
\end{lemma}

Lemma \ref{lem commu rel D} implies that $H_D(X)$ is equipped with a natural $\widetilde{W}$-action via the map $\Delta$.

\subsection{The operator $D_{\mathbb{A}}$} \label{ss operator aw}
Let $\mathbb{A}_W:=\mathbb A_W(V)=S(V)  \rtimes W$ be the skew group ring, which is equipped with a $\mathbb{Z}$-graded structure given by $\mathrm{deg}(v)=1$ for $v \in V \subset S(V)$ and $\mathrm{deg}(w)=0$ for $w\in W$. Define $D_{\mathbb{A}} \in \mathbb{A}_W \otimes C(V)$ as:
\[   D_{\mathbb{A}} = \sum_{i=1}^n \epsilon_i \otimes c_{\epsilon_i} .
\]

It is straightforward to compute that $D_{\mathbb{A}}^2=-\sum_{i=1}^n \epsilon^2 \otimes 1$. If $\overline{X}$ is a graded $\mathbb{A}_W$-module, then $D_{\mathbb{A}}^2$ acts identically by zero on $\overline{X}$. For an $\mathbb{A}_W$-module $\overline{X}$, we define similarly
\begin{eqnarray} \label{eqn dirac coh grad}
   H_{D_{\mathbb{A}}}(\overline{X}) =\frac{\mathrm{ker}~ D_{\mathbb{A}}}{\mathrm{ker}~ D_{\mathbb{A}} \cap \mathrm{im}~ D_{\mathbb{A}}} ,
\end{eqnarray}
where we realize $D_{\mathbb{A}}$ to be the operator on $\overline X$ from its $\mathbb{A}_W$-action on $\overline{X}$.

We define $\Delta_{\mathbb{A}}:\mathbb{C}[\widetilde{W}] \rightarrow \mathbb{A}_W \otimes S$ given by $\widetilde{s}_{\alpha} \mapsto t_{s_{\alpha}} \otimes \widetilde{s}_{\alpha}$. We still have a version of Lemma \ref{lem commu rel D}.
\subsection{Six-term exact sequence} \label{ss six term}

Let $X$ be an $\mathbb{H}$-module admitting a $\mathbb{Z}_2$-grading. Let $\widetilde{X}^{\pm}=X^{\pm}\otimes S$. Note that $D(X^{\pm}\otimes S) \subset X^{\mp} \otimes S$. Define $D^{\pm}=D|_{X^{\pm}\otimes S}$. 

We define
\[ H^{\pm}_D(X) =\mathrm{ker} D^{\pm}/(\mathrm{ker} D^{\pm} \cap \mathrm{im} D^{\mp}) .
\]
We remark that the notion here does not coincide with the one in \cite{CT} since we use a different $\mathbb{Z}_2$-grading on $X\otimes S$. 

\begin{definition} \label{def admit cc}
An $\mathbb{H}$-module $X$ is said to {\it admit a central character} if $X$ is finite-dimensional and any element $z$ in $Z(\mathbb{H})$ acts by a scalar on $X$. By Schur's lemma, any irreducible module always admits a central character. It is known that $Z(\mathbb{H})=S(V)^W \times \mathbb{C}[\mathbf r]$ \cite[6.5]{Lu1}. We shall identify naturally the set of central characters $\chi_{(W\gamma^{\vee},r)}$ with $(W\gamma^{\vee},r) \in (V^{\vee}/W )\times \mathbb{C}$ such that $\chi_{(W\gamma^{\vee},r)}(z,r')=\gamma^{\vee}(z)rr'$.
\end{definition}

\begin{proposition} \label{prop six term es}
Let $X$ be a $\mathbb{Z}_2$-graded $\mathbb{H}$-module admitting a central character. Suppose there exists a short exact sequence of $\mathbb{H}$-modules for $X$:
\[  0 \rightarrow X \rightarrow Y \rightarrow Z \rightarrow 0 .\]
Then we have the following six-term short exact sequence:
\[\xymatrix{ 
     & H^+_D(X)   \ar[r] & H^+_D(Y)       \ar[r] &  H^+_D(Z)  \ar[d]  &   \\
     & H^-_D(Z)  \ar[u] & \ar[l]  H^-_D(Y) &  \ar[l] H^-_D(X) & } .
\]
The statement still holds if we replace $\mathbb{Z}_2$-graded $\mathbb{H}$-modules by $\mathbb Z$-graded $\mathbb{A}_W$-modules and replace $H_D$ by $H_{D_{\mathbb{A}}}$.
\end{proposition}

For a proof of Proposition \ref{prop six term es} and the definition of the connecting homomorphisms, one may refer to a similar setting in \cite{HP}. By definitions, we have $H_D(A)=H^+_D(A)\oplus H^-_D(A)$. In particular, we have:
\begin{corollary}
\begin{enumerate}
\item If $H_D(X)=H_D(Z)=0$, then $H_D(Y)=0$.
\item If $H_D(Y)=0$, then $H_D(X) \cong H_D(Z)$. 
\end{enumerate}
\end{corollary}

\subsection{Iwahori-Matsumoto involution} \label{ss IM invol}

\begin{definition}
The Iwahori-Matsumoto $IM$ is a linear involution $IM: \mathbb{H} \rightarrow \mathbb{H}$ characterized by 
\[  IM(v)=-v \quad \mbox{ for $v \in V$ }, \quad IM(t_w)=(-1)^{l(w)}t_w .
\]
For an $\mathbb{H}$-module $X$, $IM(X)$ is an $\mathbb{H}$-module isomorphic to $X$ as vector spaces with $\mathbb{H}$-action given by:
\[   \pi_{IM(X)}(h)x=\pi_X(IM(h))x ,
\]
where $\pi_{IM(X)}$ (resp. $\pi_X$) are the maps defining the $\mathbb{H}$-action.
\end{definition}

We shall implicity use the following fact:
\begin{lemma} \label{lem iwahoi mat dirac}
Let $X$ be an $\mathbb{H}$-module. Then $H_D(X) \neq 0$ if and only if $H_D(IM(X)) \neq 0$. 
\end{lemma}
\begin{proof}
Striaghtforward from definitions.
\end{proof}
\section{Deformation for Dirac cohomology} \label{s def dirac coh}

\subsection{Deformation for modules} \label{ss def mod}

Let $S^{\leq i}(V)$ be the space containing polynomials of degree less than or equal to $i$. Let $\mathbb{H}_i=\left\{ t_wp: w \in W, p \in S^{\leq i}(V) \right\}$ for $i \geq 0$ and let $\mathbb{H}_{-1}=0$. We have a natural linear isomorphism:
\begin{eqnarray} \label{eqn iso graded algebra}
  \mathbb{A}_W \cong \bigoplus_{i \geq 0} \mathbb{H}_i/\mathbb{H}_{i-1} .
\end{eqnarray}

Let $\sigma$ be a $W$-representation. Define a $W$-filtration on $\mathbb{H}\otimes_{\mathbb{C}[W]}\sigma$ by
\[  F(\sigma)_i := \left\{ p \otimes u : u \in \lambda \mbox{ and } p \in S^{\leq i}(V)  \right\}.
\]

\begin{definition} \label{def deform structure}
Let $X=X^+ \oplus X^-$ be a $\mathbb{Z}_2$-graded $\mathbb{H}$-module. Let $\sigma$ be a $W$-subspace of $X^+$ or $X^-$. Then we define a $W$-filtration, depending on $\sigma$, on $X$ as follows. Define a natural map $\psi_{\sigma}: \mathbb{H} \otimes_{\mathbb{C}[W]} \sigma \rightarrow X$ given by $1\otimes x \mapsto x$. If $\psi_{\sigma}$ is surjective, we shall say that $\sigma$ is a {\it choice of deformation} for $X$. If $X$ is irreducible, the PBW basis for $\mathbb{H}$ implies that any $W$-subspace $\sigma$ of $X^{\pm}$ is a choice of deformation for $X$.

We now assume $\sigma$ is a choice of deformation for $X$. Let $X_{\sigma, i}=\psi_{\sigma}(F(\sigma)_i)$ or simply $X_i$ if the choice of $\sigma$ is clear. Define $\overline{X}_{\sigma}= \bigoplus_i X_{i}/X_{i-1}$. There is a natural graded $\mathbb{A}_W$-structure on $\overline{X}_{\sigma}$ which can be explicitly described as follows: Denote by $\mathrm{pr}_{\sigma,i}: X_{i} \rightarrow X_{i}/X_{i-1} \subset \overline{X}_{\sigma}$ a natural projection map. (We simply write $\mathrm{pr}_i$.) For $\mathrm{pr}_{i+1}(x) \in X_{i+1}/X_{i}$,
\begin{enumerate}
\item for $v \in V \subset \mathbb{A}_W$, $v .\mathrm{pr}_{i+1}(x) =\mathrm{pr}_{i+2}(\widetilde{v}.x) \in X_{i+2}/X_{i+1}$, where $\widetilde{v}$ acts by the $\mathbb{H}$-action on $x$;
\item for $t_w \in \mathbb{C}[W] \subset \mathbb{A}[W]$, $t_w.\mathrm{pr}_{i+1}(x)=\mathrm{pr}_{i+1}(t_w.x) \in X_{i+1}/X_i$, where $t_w.x$ is given by the $\mathbb{H}$-action on $X$.
\end{enumerate}
\end{definition}

\begin{lemma} 
The above $\mathbb{A}_W$-structure on $\overline{X}_{\sigma}$ is well-defined.
\end{lemma}
\begin{proof}
By the definitions of $X_i$, the action is independent of a choice of a representative for $\mathrm{pr}_{i+1}(x) \in X_{i+1}/X_i$. We also have to check the defining relations for $\mathbb{A}_W$.  For $v_1, v_2 \in V$, a direct computation gives that $\widetilde{v}_2\widetilde{v}_1-\widetilde{v}_1\widetilde{v}_2 \in \mathbb{C}[W]$ and hence $(\widetilde{v}_1\widetilde{v}_2-\widetilde{v}_2\widetilde{v}_1).x \in X_{i+1}$. Thus $\mathrm{pr}_{i+3}((\widetilde{v}_1\widetilde{v}_2-\widetilde{v}_2\widetilde{v}_1).x)=0$ as an element in $X_{i+3}/X_{i+2}$.

Now we consider $w \in W$ and $v\in V$. It is again from direct computation that $t_w\widetilde{v}=\widetilde{w(v)}t_w$. Thus $\mathrm{pr}_{i+2}((t_w\widetilde{v}-\widetilde{w(v)}t_w).x)=0$ as desired.
\end{proof}

By switching $X^+$ and $X^-$ if necessary, we shall assume that $\sigma \subset X^+$ from now on. Inductively, we can fix a $W$-map $\iota_{\sigma,i}: X_i/X_{i-1} \rightarrow X_i \subset X$ such that
\begin{enumerate}
\item[(1)] $\mathrm{pr}_{i} \circ \iota_{\sigma,i}(\bar{y})=\bar{y}$ for any $\bar{y} \in X_{i+1}/X_i$;
\item[(2)] $\mathrm{im}~\iota_{\sigma,i} \subset X^+$ if $i\equiv 0$ (mod $2$); and $\mathrm{im}~\iota_{\sigma,i} \subset X^-$ if $i \equiv 1$ (mod $2$). 
\end{enumerate}
There is no canonical choice for $\iota_{\sigma,i}$ in general but it is not important in our study.

 Then we obtain a $W$-map $\iota_{\sigma}=\bigoplus_i \iota_{\sigma, i}: \overline{X}_{\sigma} \rightarrow X$. We simply write $\iota$ for $\iota_{\sigma}$ if there is no confusion. As a $W$-map, define $\kappa_{\sigma}=\iota_{\sigma}^{-1} :X \rightarrow \overline{X}_{\sigma} $. We simply write $\kappa$ for $\kappa_{\sigma}$.

If we identify $X=\bigoplus_{i} \mathrm{im}~\iota_{\sigma,i}$, then we have $\kappa = \bigoplus_i \mathrm{pr}_i|_{\mathrm{im}~\iota_{\sigma,i}}$.

\begin{example}
Let $W$ be the Weyl group of type $A_1$. Let $V=\mathbb{C}\otimes_{\mathbb Z}R$. Let $\mathbb{C}x$ be a $1$-dimensional $\mathbb{H}$-module such that $v.x=0$ for all $v \in V$. Then $\mathbb{H}\otimes_{S(V)}\mathbb{C}x$ has $1$ copy of the trivial representation $\mathrm{triv}$ and $1$ copy of the sign representation $\mathrm{sgn}$. Choose $\sigma$ to be the $\mathrm{triv}$-isotypic component. Let $0\neq u \in \sigma$. Then $\overline{X}_{\sigma}$ has a basis $\left\{ u, v_0.u \right\}$ for some non-trivial $v_0\in V$. The $\mathbb{A}_w$-structure is given by $v_0(v_0.u)=0$ and $t_w.u=u$ and $t_w.u=(-1)^{l(w)}u$ for $w \in S_2$. 
\end{example}

We remark that we essentially use the $\mathbb{Z}_2$-grading in the following lemma:
\begin{lemma} \label{lem deformed action1}
For $v \in V$ and $x \in \mathrm{im}~ \iota_{\sigma, i}$,
\[  \kappa(\widetilde{v}.x)-v.\kappa(x)  \in \bigoplus_{j<i} X_j/X_{j-1} .\]
\end{lemma}

\begin{proof}
If $\mathrm{im}~\iota_{\sigma, i} \subset X^{\pm}$, then $\widetilde{v}.x \in X^{\mp}$ and in particular, $\widetilde{v}.x \notin \mathrm{im}~\iota_{\sigma,i}$. Thus we can write $\widetilde{v}.x=x_1 +y$ for some $x_1 \in \mathrm{im}~\iota_{\sigma,i+1}$ and $y \in \bigoplus_{j<i}\mathrm{im}~\iota_{\sigma,j}$. Thus it suffices to show that $\kappa(x_1)=v.\kappa(x)$. To this end, we first note that $\mathrm{pr_{i+1}}(\widetilde{v}.x)=\mathrm{pr}_{i+1}(x_1)=\kappa(x_1)$. By the definition of $\mathbb{A}_W$-action on $\overline{X}_{\sigma}$, we have $\mathrm{pr}_{i+1}(\widetilde{v}.x)=v.\mathrm{pr}_i(x)$. Hence $v.\mathrm{pr}_i(x)=\mathrm{pr}_{i+1}(x_1)=\kappa(x_1)$ as desired. 
\end{proof}

$\kappa$ induces a natural map $\widetilde{\kappa}: X \otimes S \rightarrow \overline{X}_{\sigma}\otimes S$ given by $\widetilde{\kappa}(x \otimes s)=\kappa(x) \otimes s$. Similarly, $\iota$ induces a natural map  $\widetilde{\iota}: \overline{X}_{\sigma}\otimes S \rightarrow X \otimes S$ determined by $\widetilde{\iota}(\overline{y} \otimes s)=\iota(\overline{y}) \otimes s$. We also have an analogous map $\widetilde{\iota}_{\sigma, k}: X_k/X_{k-1}\otimes S \rightarrow X \otimes S$, and shall simply write $\widetilde{\iota}_k$.

\begin{lemma} \label{lem deformed action dirac}
For $v \in V$ and $\widetilde{x} \in \mathrm{im}~ \widetilde{\iota}_{\sigma, i}$,
\[  \widetilde{\kappa}(D.\widetilde{x})-D_{\mathbb{A}}.\widetilde{\kappa}(\widetilde{x})  \in \bigoplus_{j<i} X_j/X_{j-1} \otimes S.\]
\end{lemma}

\begin{proof}
This follows from Lemma \ref{lem deformed action1} and the definitions of $D$ and $D_{\mathbb{A}}$.
\end{proof}

\begin{lemma} \label{lem minor kernel incl}
$\widetilde{\kappa}(\mathrm{ker}~D) \subset \mathrm{ker}~D_{\mathbb{A}}$
\end{lemma}
\begin{proof}
This follows from Lemma \ref{lem deformed action dirac} and the fact that $D_{\mathbb{A}}.\widetilde{\kappa}(\widetilde{x}) \in X_{i+1}/X_i\otimes S$.
\end{proof}

\begin{lemma} \label{lem tilde commute}
For $\alpha \in R^+$ and $\widetilde{x} \in X\otimes S$, $\Delta_{\mathbb{A}}(\widetilde{s}_{\alpha}).\widetilde{\kappa}(\widetilde{x})=\kappa(\Delta(\widetilde{s}_{\alpha}).\widetilde{x})$ 
\end{lemma}

\begin{proof}
This follows from the fact that $\kappa$ is a $W$-map.
\end{proof}

\subsection{Dirac cohomology}

If $X$ is a $\mathbb{H}$-module admitting a central character, then $\mathrm{ker}~D^2$ acts by a scalar on each $\widetilde{W}$-isotypic component of $X$ by using the formula (\ref{eqn formul d2}) and Lemma \ref{lem center of tilde w}. We shall implicitly use this fact below.

\begin{lemma} \label{lem minor D1}
Let $X$ be a $\mathbb{Z}_2$-graded $\mathbb{H}$-module admitting a central character. Let $\widetilde{x} \in \overline{X}_{\sigma} \otimes S$. Let $\widetilde{\tau}$ be an irreducible genuine $\widetilde{W}$-representation. Then $D_{\mathbb{A}}\widetilde{x}$ lies in the $\widetilde{\tau}$-isotypic component of $\overline{X}_{\sigma}$ if and only if $\widetilde{x}$ lies in the $\mathrm{sgn}\otimes \widetilde{\tau}$-isotypic component of $\overline{X}_{\sigma}$.
\end{lemma}
\begin{proof}
This follows from an analogous version of Lemma \ref{lem center of tilde w}.
\end{proof}

\begin{lemma} \label{lem minor D2}
Let $X$ be a $\mathbb{Z}_2$-graded $\mathbb{H}$-module admitting a central character. Let $\widetilde{\sigma}$ be an irreducible $\widetilde{W}$-representation. Then $D^2$ acts by zero on $\widetilde{\sigma}$-isotypic component of $X$ if and only if $D^2$ acts by zero on $\mathrm{sgn}\otimes \widetilde{\sigma}$-isotypic component for $X$.
\end{lemma}

\begin{proof}
This follows from the formula of $D^2$ (\ref{eqn formul d2}) and Lemma \ref{lem center of tilde w}.
\end{proof}

\begin{theorem} \label{prop vanishing deform}
Let $X$ be a $\mathbb{Z}_2$-graded $\mathbb{H}$-module admitting a central character. Let $\sigma$ be a choice of deformation for $X$. Suppose $H_{D_{\mathbb{A}}}(\overline{X}_{\sigma})=0$. Then $H_D(X)=0$. 
\end{theorem}

\begin{proof}
Let $K=\mathrm{ker}~D^2$. We shall proceed by an inductive argument. By switching $X^+$ and $X^-$ if necessary, we assume $\sigma \subset X^+$. We fix $\widetilde{\iota}$ and then obtain $\widetilde{\kappa}$ as in Section \ref{ss def mod}.

For simplicity, let $\widetilde{X}_i=(X_i  \otimes S)\cap K$. Let $k$ be the least integer such that $\widetilde{X}_k \neq 0$. Let $\widetilde{x} \in \widetilde{X}_k \cap \mathrm{ker}~D$. By Lemma \ref{lem minor kernel incl}, we have $D_{\mathbb{A}}.\widetilde{\kappa}(\widetilde{x})=0$. Suppose $\widetilde{x} \neq 0$ and hence $\widetilde{\kappa}(\widetilde{x})\neq 0$. Then there exists $\bar{y} \in  X_{k-1}/X_{k-2}$ such that $D_{\mathbb{A}}.\bar{y}=\widetilde{\kappa}(\widetilde{x})$ by $H_{D_{\mathbb{A}}}(\overline{X}_{\sigma})=0$. However by Lemmas \ref{lem tilde commute}, \ref{lem minor D1} and \ref{lem minor D2}, we have $D^2$ acts by zero on a non-zero element $\widetilde{\iota}(\bar{y})$, contradicting the minimality of our choice of $k$. Hence, $\widetilde{x}=0$ and so $\widetilde{X}_k\cap \mathrm{ker}~D=0$.

 Let $\widetilde{Y}^0=\widetilde{X}_k  \oplus D(\widetilde{X}_k) \subset K$.  From above discussion, we have $\mathrm{ker} D|_{\widetilde{Y}^0}=\mathrm{ker} D|_{\widetilde{Y}^0} \cap \mathrm{im} D|_{\widetilde{Y}^0}$. 


 We now proceed to consider the least $k' (\geq k+1)$ such that $\widetilde{X}_{k'}/\widetilde{Y}^0 \neq 0$. Let $\widetilde{x}' \in (\widetilde{X}_{k'} \cap \mathrm{im}~\widetilde{\iota}_{k'})\setminus \widetilde{Y}^0$. We claim that $D. \widetilde{x}' \notin \widetilde{Y}^0$. Otherwise $D_{\mathbb A}.\widetilde{\kappa}(\widetilde{x}')=0$ using a similar argument in the previous paragraph and the proof of Lemma \ref{lem minor kernel incl}. Now $H_{\mathbb{D}_{\mathbb{A}}}(\overline{X}_{\sigma})=0$ implies there exists $\bar{y}'\in \bigoplus_{j<k'} X_j/X_{j-1}$ such that $D_{\mathbb{A}}\bar{y}'=\widetilde{\kappa}(\widetilde{x}')$. By our minimality of choice of $k'$ (with Lemmas \ref{lem tilde commute}, \ref{lem minor D1} and \ref{lem minor D2}), we have $\iota(\bar{y}') \in \widetilde{Y}^0$. Now consider $\widetilde{\kappa}(D.\widetilde{\iota}(\bar{y}'))-D_{\mathbb{A}}.\widetilde{\kappa}(\widetilde{\iota}(\bar{y}'))=\widetilde{\kappa}(D.\widetilde{\iota}(\bar{y}'))-\widetilde{\kappa}(\widetilde{x}')$, we have a contradiction to Lemma \ref{lem deformed action dirac} (since $\widetilde{\kappa}(\widetilde{x})$ in $X_{k'}/X_{k-1}' \otimes S $ but not in $\widetilde{Y}^0$). Thus we proved our claim $D.\widetilde{x} \notin \widetilde{Y}^0$. 

Now consider $\widetilde{Y}^1=\widetilde{X}_{k'} \oplus D(\widetilde{X}_{k'})$. From the proved claim, we again have $\mathrm{ker} D|_{\widetilde{Y}^1}=\mathrm{ker} D|_{\widetilde{Y}^1} \cap \mathrm{im} D|_{\widetilde{Y}^1}$. 


Repeating the process above and using the finite-dimensionality, we ultimately have that $\ker~D=\mathrm{im}~D \cap \mathrm{ker}~D$ as desired.
\end{proof}

\begin{corollary} \label{cor dirac cohom def1}
Let $X$ be a $\mathbb{Z}_2$-graded finite-dimensional $\mathbb{H}$-module admitting a central character. Set $X \otimes S =\widetilde{X}_1 \oplus \widetilde{X}_2$ as a decomposition of $\widetilde{W}$-representations such that $D(\widetilde{X}_k) \subset \widetilde{X}_k$ and $D_{\mathbb{A}}(\widetilde{\kappa}(\widetilde{X}_k))\subset \widetilde{\kappa}(\widetilde{X}_k)$ ($k=1,2$). Furthermore, assume that the decomposition satisfies the property that $X_i \otimes S=(\widetilde{X}_1 \cap (X_i \otimes S)) \oplus (\widetilde{X}_2 \cap (X_i \otimes S))$ for all $i$. Then if 
\[ \frac{\mathrm{ker} D_{\mathbb{A}}|_{\widetilde{\kappa}(\widetilde{X}_2)}}{\mathrm{ker} D_{\mathbb A}|_{\widetilde{\kappa}(\widetilde{X}_2)}\cap \mathrm{im} D_{\mathbb A}|_{\widetilde{\kappa}(\widetilde{X}_2)}} =0
\]
 then  $H_{\mathbb{D}}(X)$ is isomorphic to a quotient of $\widetilde{X}_1$ as $\widetilde{W}$-representations.
\end{corollary}

\begin{proof}
Using the assumptions, we have
\[  H_D(X) \cong \frac{\mathrm{ker} D|_{\widetilde{X}_1}}{\mathrm{ker} D|_{\widetilde{X}_1}\cap \mathrm{im} D|_{\widetilde{X}_1}} \oplus \frac{\mathrm{ker} D|_{\widetilde{X}_2}}{\mathrm{ker} D|_{\widetilde{X}_2}\cap \mathrm{im} D|_{\widetilde{X}_2}}  .
\]
Using the argument in Theorem \ref{prop vanishing deform}, the latter summand in the right hand side is zero. Then the corollary follows.
\end{proof}

\subsection{Deformation for unitary modules}

Certain deformation for unitary modules behaves nicely since we have $H_D(X)=\mathrm{ker}~D=\mathrm{ker}~D^2$. 

We recall a well-known property for the Dirac cohomology of unitary modules, which can be proven by linear algebra and the property that $D^*=D$.

\begin{proposition}
Let $X$ be a unitary $\mathbb{H}$-module. Then $H_D(X)=\mathrm{ker}~D$. 
\end{proposition}

\begin{proposition} \label{prop unitary dirac}
Let $X$ be an irreducible unitary ($\mathbb{Z}_2$-graded) $\mathbb{H}$-module. Let $\sigma$ be a choice of deformation for $X$ with $\sigma \subset X^+$. Then $H_{D_{\mathbb{A}}}(\overline{X}_{\sigma}) \cong H_D(X)$ .
\end{proposition}
\begin{proof}
 Since $D^2$ acts as a diagonal matrix on $X$ and $\mathrm{ker} D=\mathrm{ker} D^2$, we have a $\widetilde{W}$-decomposition $X \otimes S =\mathrm{ker}~D \oplus  \bigoplus_{\lambda\neq 0} \widetilde{X}_{\lambda}$, where $\widetilde{X}_{\lambda}= \mathrm{ker}(D^2-\lambda)$. $\widetilde{X}_{\lambda}$ is compatible with the filtration $\left\{ X_i \right\}$ on $X$ i.e. $X_i \otimes S=(\ker D \cap (X_i \otimes S)) \oplus\bigoplus_{\lambda\neq 0}  (\widetilde{X}_{\lambda} \cap (X_i \otimes S))$ since $D^2(X_i\otimes S) \subset X_i \otimes S$.  

{\it Claim:} $\mathrm{ker} D_{\mathbb{A}}|_{\kappa(\widetilde{X}_{\lambda})}=\mathrm{im} D_{\mathbb{A}}|_{\kappa(\widetilde{X}_{\lambda})} \cap \mathrm{ker} D_{\mathbb A}|_{\kappa(\widetilde{X}_{\lambda})}$ for each $\lambda \neq 0$.

{\it Proof of claim:}  Let $\sum_{r} \kappa(\widetilde{x}_r) \in \mathrm{ker} D_{\mathbb A}|_{\widetilde{X}_{\lambda}}$ be an element satisfying $\widetilde{x}_r \in \iota(X_r/X_{r-1}) \otimes S$. By a grading consideration, we have $\widetilde{\kappa}(\widetilde{x}_r) \in \mathrm{ker} D_{\mathbb A}$ for each $r$. Then $D( \widetilde{x}_r ) \in X_{r-1} \otimes S$ by Lemma \ref{lem minor kernel incl}. 

 Let $y=D.\widetilde{x}_r$. We have $D.y=\lambda\widetilde{x}_r$ for some $\lambda \neq 0$. By Lemma \ref{lem deformed action dirac}, $D_{\mathbb{A}}.\kappa(y) =\lambda \kappa(\widetilde{x}_r)$. Hence we obtain $\kappa(\widetilde{x}_r) \in \mathrm{im}~D_{\mathbb{A}}$ for each $r$. Hence we have $\mathrm{ker} D_{\mathbb{A}}|_{\widetilde{\kappa}(\widetilde{X}_{\lambda})}=\mathrm{im} D_{\mathbb{A}}|_{\widetilde{\kappa}(\widetilde{X}_{\lambda})} \cap \mathrm{ker} D_{\mathbb A}|_{\widetilde{\kappa}(\widetilde{X}_{\lambda})}$. This proves the claim.

 Now with Lemma \ref{lem minor kernel incl}, we obtain the statement.
\end{proof}
We shall prove that a unitary $\mathbb{H}$-module can extend to a $\mathbb{Z}_2$-graded module in Lemma \ref{lem unitary z2 grade}. 
\section{Structure of tempered modules and generalized Springer correspondence} \label{s struct temp mod}

This section reviews results in \cite{Lu1, Lu2, Lu3, Ka}  and our goal is to obtain Corollary \ref{cor struct h}.

\subsection{Graded Hecke algebras of geometric parameters} \label{ss geo gha}

Let $G$ be a complex connected reductive algebraic group with the Lie algebra $\mathfrak{g}$.

\begin{definition}\cite{Lu2}
A {\it cuspidal triple} for $G$ is a triple $(L, \mathcal O, \mathcal L)$ such that $L$ is a Levi subgroup of $G$, $\mathcal O$ is a nilpotent $L$-orbit in the Lie algebra $\mathfrak{l}$ of $L$, and $\mathcal L$ is a $L$-equivariant cuspidal local system for $\mathcal O$ in the sense of \cite{Lu0}.
\end{definition}

Fix a cuspidal triple $(L, \mathcal O, \mathcal L)$.  Let $T=Z_L^{\circ}$, where $Z_L$ is the centralizer of $L$. Let $\mathfrak{h}$ be the Lie algebra of $T$.  Let $\mathfrak{h}^{\vee}$ be the dual space of $\mathfrak{h}$. Let $W=N_G(T)/L$, where $N_G(T)$ is the normalizer of $T$ in $G$. The root decomposition of $\mathfrak{g}$ by $\mathfrak{h}$ defines a root system $R \subset \mathfrak{h}^{\vee}$. 

Let $P=LU$ be a choice of parabolic subgroups with the Levi part $L$.  We have a projection map:
\[ \mu: G \times_P (\overline{\mathcal O}\oplus \mathfrak{u}) \rightarrow \mathfrak{g}
 \]
given by $\mu(g, x)=\mathrm{ad}(g)(u)$. Denote the image of $\mu$ by $\mathcal N$. Let $j: \mathcal O \hookrightarrow \overline{\mathcal O}$ be the inclusion and let $\mathrm{pr}: \overline{\mathcal O} \oplus \mathfrak{u} \rightarrow \overline{\mathcal O}$ be the projection map. By the cleanness property, we have $j_*\mathcal L \cong j_1\mathcal L$ and hence $\mathrm{pr}^*j_!\mathcal L$ defines a $P \times \mathbb{C}^*$-equivariant sheaf on $\overline{\mathcal O} \oplus \mathfrak{u}$. Here $P\times \mathbb{C}^{\times}$ acts on $\overline{\mathcal O} \oplus \mathfrak{u}$ by the action $(g,r).x=r^{-2}\mathrm{ad}(g)(x)$. We then obtain a corresponding $G \times \mathbb C^{\times}$-equivariant sheaf $\dot{\mathcal L}$ on $G\times_P(\overline{\mathcal O}\oplus \mathfrak{u})$.  

The parabolic subgroup $P$ determines a set $\Pi$ of simple roots in $R$. We can obtain a parameter function $c':\Pi \rightarrow \mathbb{R}$ as in \cite[8.7]{Lu2}. We set $c=\frac{1}{2}c'$. Such normalization is for the consistency of the computation in \cite{BCT} later. 
 Lusztig \cite[Sec. 8]{Lu2} proves that there are canonical isomorphisms
\[  \mathrm{Hom}_{D^b_{G \times \mathbb{C}^{\times}}(\mathcal N)}(\mu_*\dot{\mathcal L}, \mu_*\dot{\mathcal L}) \cong \mathbb{C}[W],\quad  \mathrm{Ext}^{\bullet}_{D^b_{G \times \mathbb{C}^{\times}}(\mathcal N)}(\mu_*\dot{\mathcal L}, \mu_*\dot{\mathcal L}) \cong \mathbb{H},  \]
where $\mathbb{H}=\mathbb{H}(\mathfrak{h}^{\vee}, \Pi, W,  \mathbf r, c)$ (Definition \ref{def gaha}). 

We have the following commutative diagram (see \cite[7.14, 8.10(a)]{Lu2}):
\[\xymatrix{ 
  &  \mathrm{Hom}_{D^b_{G\times \mathbb{C}^{\times}}(\mathcal N)}(\mu_* \dot{\mathcal L},\mu_*\dot{\mathcal L}) \ar[d] \ar[r] &  \mathbb{C}[W] \ar[d]^{}  \\
 &   \mathrm{Ext}^{\bullet}_{D^b_{G \times \mathbb{C}^{\times}}(\mathcal N)}(\mu_* \dot{\mathcal L}, \mu_*\dot{\mathcal L}) \ar[r]  & \mathbb{H}  } ,
\]
where the right vertical map is the injection given by $w \mapsto t_w$ and the left vertical map is given \cite[7.14, 8.10(a)]{Lu2}. Here $D^b_{G\times \mathbb C^{\times}}(\mathcal N)$ is the $G\times \mathbb{C}^{\times}$-equivariant derived category. 

The geometric parameters for graded Hecke algebras are described in \cite[Section 2.13]{Lu2}. In particular, when $G$ is simply connected and almost simple, $L$ is the maximal torus and $\mathcal O$ is the zero orbit with a constant sheaf $\mathcal L$, we have $c \equiv 1$.


\begin{lemma} \label{lem unitary z2 grade}
Assume $c_{\alpha}\neq 0$ for all $\alpha \in \Pi$. Let $X$ be an irreducible unitary $\mathbb{H}$-module with $\mathbf r$ acting by a real scalar and with all weights being real. Then $X$ can extend to a $\mathbb{Z}_2$-graded module. 
\end{lemma}

\begin{proof}
By definitions, $X \cong X^*$. We also have $X^* \cong \theta(X)$ by \cite[Lemma 4.5]{Ch1.5} and \cite[Lemma 2.5]{Ch5}. Hence $X \cong \theta(X)$. The lemma now follows from Lemma \ref{lem z2 grading}. 
\end{proof}

\subsection{Generalized Springer correspondence}

Let $e \in \mathcal N$. Let $Z_{G\times \mathbb{C}^{\times}}(e)$ (resp. $Z_G(e)$) be the centralizer of $e$ in $G\times \mathbb{C}^{\times}$ (resp. in $G$). Denote by $\mathcal O_e$ the $G$-orbit of $e$ in $\mathcal N$. Let $A(e)=Z_{G\times \mathbb{C}^{\times}}(e)/Z_{G\times \mathbb{C}^{\times}}(e)^{\circ} \cong Z_G(e)/Z_G(e)^{\circ}$ be the component group. Let $\mathfrak{B}_e=\mu^{-1}(e)$. Define $H_{\bullet}(\mathfrak{B}_e,\dot{\mathcal L})=\mathbb{H}^{\bullet}(i_e^!\mu_*\dot{\mathcal L}[2\mathrm{dim}\mathcal N-2\mathrm{dim} \mathfrak{B}_e])$ as in \cite[Theorem 3.1]{Ka} (also see \cite[Remark 3.2]{Ka}), where $i_e: \left\{e \right\} \rightarrow \mathcal N$ is the natural inclusion map. The space $H_{\bullet}(\mathfrak{B}_e, \dot{\mathcal L})$ is equipped with a natural $W$-structure and $A(e)$-action which commutes with the $W$-structure \cite{Lu1}, \cite[Theorem 3.1(6)]{Ka}. For $e \in \mathcal N$ and $\zeta \in \mathrm{Irr} A(e)$, let 
\[  K_{e, \zeta}= H_{\bullet}(\mathfrak{B}_e, \dot{\mathcal L})_{\zeta} := \mathrm{Hom}_{A(e)}(\zeta, H_{\bullet}(\mathfrak{B}_e, \dot{\mathcal L})),
\]
regarded as a $W$-representation. Let $\mathcal P$ be the collection all pairs of $(\mathcal O, \zeta)$ such that $\mathcal O$ is a nilpotent $G$-orbit in $\mathcal N$, $\zeta \in \mathrm{Irr}A(e)$ ($e \in \mathcal O$) and $K_{e,\zeta}\neq 0$. We have the following upper triangulation property:

\begin{theorem} (\cite{Lu0}, \cite[Theorem 3.1]{Ka}) \label{thm w struct spr}
There is a bijective map, denoted $\Phi: \mathcal P \rightarrow \mathrm{Irr} W$, such that  for any $(\mathcal O, \zeta) \in \mathcal P$
\[  \mathrm{Hom}_W(\Phi(\mathcal O, \zeta), K_{e', \zeta'}) =0
\] 
if $\mathcal O_{e'} \not\subseteq \overline{\mathcal O}$ or $\zeta\neq \zeta'$ (for $\mathcal O=\mathcal O_{e'}$); and when $\mathcal O_{e'}=\mathcal O$ and $\zeta=\zeta'$, $\mathrm{dim}~\mathrm{Hom}_W( \Phi(\mathcal O, \zeta),K_{e', \zeta})=1$.
\end{theorem}


Let $\mathbb{A}_W=\mathbb{C}[\mathfrak{h}^{\vee}] \rtimes \mathbb{C}[W]$. It is shown in \cite[Theorem 3.1(1)]{Ka} (also see \cite[Sec. 7]{Lu2}) that $\mathbb{A}_W$ can be identified with $\mathrm{Ext}^{\bullet}_{D^b_G(\mathcal N)}(\mu_*\dot{\mathcal L}, \mu_*\dot{\mathcal L})$, where $D^b_G(\mathcal N)$ is the $G$-equivariant derived category for $\mathcal N$. One can extend the $W$-structure on $K_{e, \zeta}$ to an $\mathbb{A}_W$-module structure \cite[Theorem 3.1]{Ka}.

For $\sigma \in \mathrm{Irr}W$, define $K_{\sigma}=\mathbb{A}_W \otimes_{\mathbb{C}[W]} \sigma$ as an $\mathbb{A}_W$-module with elements in $1 \otimes \sigma$ having degree $0$. Define the partial ordering $\leq$ on $\mathcal P$ given by the closure ordering i.e. $(\mathcal O, \zeta) \leq (\mathcal O',\zeta')$ if and only if $\mathcal O \subseteq \overline{\mathcal O}'$. For $\sigma \in \mathrm{Irr}W$, define modules in \cite{Ka}:
\begin{equation} \label{eqn kato mod}
 \mathbf K_{\sigma} = K_{\sigma}/\left( \sum_{ f\in \mathrm{gHom}_{\mathbb{A}_W}(K_{\tau}, K_{\sigma})_{>0}, \tau \leq \sigma } \mathrm{im} f \right)
\end{equation}
Here $\mathrm{gHom}_{\mathbb{A}_W}( K_{\tau}, K_{\sigma})_{>0}$ is defined as follows. Let $\overline{X}=\oplus_i \overline{X}_i$ and $\overline{Y}=\oplus_i \overline{Y}_i$ be graded $\mathbb{A}_W$-modules. A linear map from $\overline{X}$ to $\overline{Y}$ is {\it graded} if $f(\overline{X}_0) \subset \overline{Y}_i$ for some $i$. Let $\mathrm{gHom}_{\mathbb{A}_W}(\overline{X}, \overline{Y})_{>0}$ be the set of all graded maps $f$ from $\overline{X}$ to $\overline{Y}$ satisfying that $f(\overline{X}_0)\not\subset \overline{Y}_0$.

\begin{theorem} \cite[Theorem 3.3]{Ka} \label{thm kato structure}
Let $(\mathcal O, \zeta) \in \mathcal P$. Let $e \in \mathcal O$. As $\mathbb{A}_W$-modules, 
\[  K_{e, \zeta} \cong \mathbf{K}_{\Phi(\mathcal O_e, \zeta)} .
\]
\end{theorem}

\subsection{Deformed structure of a tempered module} \label{ss deform struct rtm}

For $r \in \mathbb{C}^{\times}$, denote by $\mathrm{Re}(r)$ the real part of $r$. Recall that $\mathbb{C} \otimes_{\mathbb Z} R$ admits a natural real form and hence $\mathfrak{h}^{\vee}$ has a real form extending the one on $\mathbb{C}\otimes_{\mathbb{Z}}R$. We then similarly have a notion of $\mathrm{Re}(s)$ for $s \in \mathfrak{h}$.

Let $(\mathcal O,\zeta)\in \mathcal P$. The space $H_{\bullet}^{Z_{G\times \mathbb{C}^{\times}}(e)^{\circ}}( \mathfrak{B}_e, \dot{\mathcal L})$ admits an $\mathbb{H}$-action as well as a $A(e)$-action (\cite[10.11]{Lu2}). Let $\mathfrak{m} \subset \mathfrak{g} \oplus \mathbb{C}$ be the Lie algebra of $Z_{G\times \mathbb{C}^{\times}}(e)^{\circ}$.

 Define $E_{e} =  H_{\bullet}^{Z_{G\times \mathbb{C}^{\times}}(e)^{\circ}}( \mathfrak{B}_e, \dot{\mathcal L})$.

The space $H_{G\times \mathbb{C}^{\times}}^{\bullet}:=H_{G\times \mathbb{C}^{\times}}^{\bullet}(pt, \mathbb{C})$ is the center of $\mathbb{H}$ via the geometric construction in Section \ref{ss geo gha} \cite[8.13]{Lu1}. It can be canonically identified with the polynomial functions on $\mathfrak{g}\oplus \mathbb{C}$ which are invariant under the action of $G$ \cite[1.11(a)]{Lu1}. Let $(s, r)$ be a semisimple element in $\mathfrak{g} \oplus \mathbb{C}$. Define $\chi_{(s,r)}: H_{G\times \mathbb{C}^{\times}}^{\bullet} \rightarrow \mathbb{C}$ given by $\chi_{(s,r)}(f)=f(s,r)$. Let $\mathcal I_{(s,r)}$ be the maximal ideal in $H_{G\times \mathbb{C}^{\times}}^{\bullet}$ containing all functions vanishing at $(s,r)$. Let $\mathrm{Mod}_{(s,r)} \mathbb{H}$ be the category of finite-dimensional $\mathbb{H}$-modules annihilated by some powers of $\mathcal I_{(s,r)}$. 

We now further assume that $[s, e]=2re$. Then $(s,r) \in \mathfrak{m} \oplus \mathbb{C}$, where $\mathfrak{m} \oplus \mathbb{C}$ is the Lie algebra of $Z_{G\times \mathbb{C}^{\times}}(e)^{\circ}$. Hence $\chi_{(s, r)}$ descends to a map from $H_{Z_{G\times \mathbb{C}^{\times}}(e)^{\circ}}^{\bullet}$ to $\mathbb{C}$. This defines a $1$-dimensional $H_{Z_{G\times \mathbb{C}^{\times}}(e)^{\circ}}^{\bullet}$-module, denoted $\mathbb{C}_{(s,r)}$. 

We shall {\it assume $r \in \mathbb{R}^{\times}$} in the remainder of this paper. (By rescaling the action on $S(V)$, one indeed has a natural equivalene of categories $\mathrm{Mod}_{(r's,r')} \mathbb{H}\cong \mathrm{Mod}_{(rs,r)} \mathbb{H}$.) We follow the terminology of temperedness of Lusztig \cite{Lu3}:
\begin{definition} \label{def temp mod}
 An $\mathbb{H}$-module $X$ in $\mathrm{Mod}_{(s,r)} \mathbb{H}$ is said to be {\it tempered} if any $S(\mathfrak{h}^{\vee})$-weight $s \in \mathfrak{h}$ of $X$ satisfies $\mathrm{Re}(\omega_{\alpha}(s))/\mathrm{Re}(r) \geq 0$ for all $\alpha \in \Pi$. Here $\omega_{\alpha}$ is the fundamental weight associated to $\alpha$ satisfying $\beta^{\vee}(\omega_{\alpha})=\delta_{\alpha,\beta}$. 
\end{definition}

We remark that most of other references use the condition that $\mathrm{Re}(\omega_{\alpha}(s))/\mathrm{Re}(r) \leq 0$ instead of $\mathrm{Re}(\omega_{\alpha}(s))/\mathrm{Re}(r) \geq 0$ for the notion of temperedness. Those two notions are related by the Iwahori-Matsumoto involution.

\begin{definition} \label{def standard mod}
Let $Z_{G\times \mathbb{C}^{\times}}(e,s)=Z_{G\times \mathbb{C}^{\times}}(e) \cap (Z_G(s) \times \mathbb{C}^{\times})$. Define $A(e,s)=Z_{G\times \mathbb{C}^{\times}}(e,s)/Z_{G\times \mathbb{C}^{\times}}(e,s)^{\circ}$ to be the component group. Define
\[ E_{s, e,r} =\mathbb{C}_{(s,r)} \otimes_{H_{Z_{G\times \mathbb{C}^{\times}}(e)^{\circ}}^{\bullet}} H^{Z_{G\times \mathbb{C}^{\times}}(e)^{\circ}}_{\bullet}(\mathfrak{B}_e, \dot{\mathcal L}) ,\]
which again admits $\mathbb{H}$-module structure and admits $A(e,s)$-module structure. By regarding $A(e,s)$ as a subgroup of $A(e)$, the $W \times A(e,s)$-structure can be identified with the one of $H_{\bullet}(\mathfrak{B}_e, \dot{\mathcal L})$ \cite[10.13]{Lu2}.

For $\zeta \in \mathrm{Irr}A(e,s)$, define, as an $\mathbb{H}$-module,
\[ E_{s, e,r,\zeta } =\mathrm{Hom}_{A(e,s)}(\zeta, E_{s,e,r}) . \]
We shall call $E_{s,e,r,\zeta}$ is a {\it standard module}. Now according to \cite[10.8, 10.12]{Lu1}, we have $E_{s, e,r,\zeta}$ as an object in $\mathrm{Mod}_{(s,r)} \mathbb{H}$.
\end{definition}

For each $e \in \mathcal N$, let $\left\{ e,h_e,f \right\}$ be a $\mathfrak{sl}_2$-triple. Then $E_{rh_e,e,r,\zeta}$ is a tempered module by \cite[Theorem 1.21]{Lu3}. We remark that $h_e$ is not unique up to conjugation if $G$ is not semisimple.

\begin{lemma} \label{lem temp z2 grade}
Any tempered module of the form $E_{rh_e,e,r,\zeta}$ admits a $\mathbb{Z}_2$-grading.
\end{lemma}
\begin{proof}
By doing some rescaling on the action of $V \subset S(V)$ on $E_{rh_e,e,r,\zeta}$, the statement can be reduced to the case that $r=1$. Thus it suffices to show that $E_{h_e,e,1,\zeta}$ is unitary and has real weights, by using  Lemma \ref{lem z2 grading}. It is equivalent to show that the Iwahori-Matsumoto dual of $E_{h_e,e,1,\zeta}$ is unitary. To show the last statement, one can see the arguments in \cite[Pg 463, 464]{Ci0} or \cite[Sec. 7]{So}, using \cite[2.22]{Op}. For the assertion that $E_{h_2,e,1\zeta}$ has real weights, it follows from \cite[Corollary 1.18 and Theorem 1.22]{Lu3}.


\end{proof}

Since $\Phi(\mathcal O_e, \zeta)$ appears with multiplicity one in $E_{rh_e,e,r,\zeta}$ (Theorem \ref{thm w struct spr}), $\Phi(\mathcal O_e, \zeta)$  is a choice of deformation for $E_{rh_e,e,r,\zeta}$. We state our main conclusion in this section:

\begin{corollary} \label{cor struct h}
Let $(\mathcal O_e, \zeta) \in \mathcal P$. Set $E=E_{rh_e,e,r,\zeta}$. Recall that $\overline{E}_{\Phi(\mathcal O_e,\zeta)}$ is an $\mathbb{A}_W$-module defined in Definition \ref{def deform structure}. Then 
\[ \overline{E}_{\Phi(\mathcal O_e, \zeta)} \cong \mathbf{K}_{\Phi(\mathcal O_e, \zeta)} .
\] 
\end{corollary}

\begin{proof}
By the definition of $\overline{E}_{\Phi(\mathcal O_e, \zeta)}$, there is a surjective map from $\mathbb{A} \otimes_{\mathbb{C}[W]} \Phi(\mathcal O_e,\zeta)$ to $\overline{E}_{\Phi(\mathcal O_e, \zeta)}$. By considering the $W$-structure from Theorem \ref{thm w struct spr}, we have a surjective map from $\mathbf K_{\Phi(\mathcal O_e, \zeta)}$ to  $\overline{E}_{\Phi(\mathcal O_e, \zeta)}$. The statement now follows from Theorem \ref{thm kato structure}.
\end{proof}

\begin{definition} \label{def lowest w type}
We keep using above notations. We shall say that $\Phi(\mathcal O_e,\zeta)$ is the {\it lowest $W$-type} of $E_{rh_e,e,r,\zeta}$. 
\end{definition}

\section{Structure of standard modules}
This section is a continuation of Section \ref{s struct temp mod}. We shall keep using notations in Section \ref{s struct temp mod}. In particular, we fix a graded Hecke algebra $\mathbb{H}$ of geometric parameters.

\subsection{Standard modules and central characters} \label{s standard mod}


Recall that a standard module is defined in Definition \ref{def standard mod}. Using \cite[Corollary 1.18]{Lu3}, one can write a standard module in terms of a module parabolically induced from an irreducible one. Moreover, the center of $\mathbb{H}$ (see \cite[8.13]{Lu2})  also lies in the parabolic subalgebra of $\mathbb{H}$. This gives the following:

\begin{lemma} \label{lem stand admit cc}
The standard module $E_{s,e,r,\zeta}$ admits a central character (Definition \ref{def admit cc}).
\end{lemma}


\subsection{Lowest $W$-types} 





\begin{notation}
Let $e \in \mathcal N$. Let $\mathrm{Irr}_eW$ be the set of irreducible $W$-representation $\sigma$ such that $\sigma$ is a lowest $W$-type of $E_{rh_e,e,r,\zeta}$ for some $\zeta \in \mathrm{Irr} A(e)$ (see Definition \ref{def lowest w type}).
\end{notation}

\begin{definition}
Let $E_{s,e,r,\zeta}$ be a standard module. By \cite[Theorem 1.15]{Lu3}, $E_{s,e,r,\zeta}$ has a unique simple quotient. Denote such unique simple quotient by $L_{s,e,r,\zeta}$.
\end{definition}

\begin{corollary} \label{cor lowest type}
Let $E:=E_{s,e,r,\zeta}$ be a standard module as before. Set $E_T=\mathrm{Hom}_{A(e,s)}(\zeta, E_{rh_e,e,r})$. Here $A(e,s)$ is regarded as a natural subgroup of $A(e,h_e)$. Write $E_T=E_1 \oplus \ldots \oplus E_r$ as the direct sum of irreducible tempered modules. Let $\sigma_i \in \mathrm{Irr}_eW$ be the lowest $W$-type of each $E_i$. Then as a $W$-representation, $L_{s,e,r,\zeta}$ contains copies of $\sigma_i$ appearing with the same multiplicities.
\end{corollary}

\begin{proof}
We shall prove by an inductive argument. We fix a semisimple element $s \in \mathfrak{h}$ and $r \in \mathbb{R}^{\times}$. Let $\left\{ X_1,\ldots, X_k \right\}$ be a collection of mutually non-isomorphic simple modules in $\mathrm{mod}_{s,r}\mathbb{H}$. By \cite[Theorem 1.15(b)]{Lu3}, $X_i=L_{s, e_i,r,\zeta_i}$ for some $e_i \in \mathcal N$ satisfying $[s,e_i]=2re_i$ and $\zeta_i \in \mathrm{Irr}A(e,s)$. By suitable renaming, we shall assume $\mathcal O_{e_i} \not\supseteq \mathcal O_{e_j}$ if $i<j$. Denote by $p_{(e_i,\zeta_i),(e_j,\zeta_j)}$ the multiplicity of $L_{s,e_j,r,\zeta_j}$ in the composition series of $E_{s,e_i,r,\zeta_i}$. By \cite[10.8]{Lu2}, $[p_{(e_i,\zeta_i),(e_j,\zeta_j)}]_{i,j}$ forms an upper triangular matrix with $1$ in diagonal. Now with the $W$-structure of $E_{s,e_i,r,\zeta_i}$ from Theorem \ref{thm w struct spr}, we obtain the statement.
\end{proof}

\begin{definition}
In the notation of Corollary \ref{cor lowest type}, each $\sigma_i$ is called a {\it lowest $W$-type} of $E_{s,e,r,\zeta}$.
\end{definition}

\begin{lemma} \label{lem top generate}
Let $E$ be a standard module. Let $E'=E\oplus \theta(E)$ if $E\neq \theta(E)$ and let $E'=E$ if $E=\theta(E)$. Let $\sigma$ be a lowest $W$-type of $E$. Let $U_{\sigma}$ be a subspace of $E'$ such that $U_{\sigma} \cong \sigma$ and $U_{\sigma} \subset E'{}^{\pm}$. Then $U_{\sigma}$ is a choice of deformation for $E'$.
\end{lemma}
\begin{proof}
If $E$ is a standard module, then $\theta(E)$ is also a standard module. Let $L$ be the unique simple quotient of $E$. Then $\theta(L)$ is the unique simple quotient of $\theta(E)$. By the uniqueness of a simple quotient, we have $E \cong \theta(E)$ if and only if $L \cong \theta(L)$. In the case that $E \not\cong \theta(E)$, the condition $U_{\sigma}  \subset E'{}^{\pm}$ implies that $U_{\sigma} \not\subset E$ and $U_{\sigma} \not \subset \theta(E)$. With Corollary \ref{cor lowest type} and the fact that the unique simple quotient appears with multiplicity one in the standard module, we have the lemma.
\end{proof}

\subsection{Filtrations on deformed standard modules}

\begin{lemma} \label{lem filt kato mod}
Let $E=E_{s,e,r,\zeta}$ be a standard module. Set $E'$ as in Lemma \ref{lem top generate}. Set $E_T=\mathrm{Hom}_{A(e,s)}(\zeta, E_{rh_e,e,r})$ and write $E_T=E_1 \oplus \ldots \oplus E_l$ as the direct sum of irreducible tempered modules. Let $\sigma_i$ be the lowest $W$-type of each $E_i$. Set $\sigma=\sigma_1$. Then there exists an $\mathbb{A}_W$-module filtration $\left\{ \overline{Y}_i \right\}_{i=0}^p$ on $\overline{E'}_{\sigma}$ such that 
\[ 0=\overline{Y}_0 \subset \overline{Y}_1 \subset \overline{Y}_2 \subset \ldots \subset \overline{Y}_k=\overline{E'}_{\sigma} \]
and for $i=0,\ldots, p-1$, $\overline{Y}_{i+1}/\overline{Y}_i$ is isomorphic to a module $\mathbf{K}_{\sigma_j}$ for some $\sigma_j$. 
\end{lemma}

\begin{proof}
 By Lemma \ref{lem z2 grading}, $E'=E'{}^+\oplus E'{}^-$ admits a $\mathbb{Z}_2$-grading. Let $U_{\sigma}$ be an irreducible $W$-subspace of $E'$ with $U_{\sigma} \cong \sigma$ as $W$-representations and $U_{\sigma} \subset E'{}^{\pm}$. By Lemma \ref{lem top generate}, $U_{\sigma}$ is a choice for deformation. Identify $\sigma$ with $U_{\sigma}$. Recall that the $\mathbb{A}_W$-module $\overline{E'}_{\sigma}$ has a $\mathbb{Z}$-grading, and hence we obtain a $\mathbb{Z}$-graded decomposition $\overline{E'}_{\sigma}=\bigoplus_{k=1}^m \overline{E'}^k$. Since $\mathrm{deg}(t_w)=0$, $W$ acts on each $\overline{E'}^k$ 

Let $Z$ be the sum of all $\sigma_i$-isotypic components in $\overline{E'}_{\sigma}$, where $i$ runs for $1,\ldots, p$. We decompose $Z=\bigoplus_{l=1}^pP_l$ into irreducible $W$-spaces compatible with the $\mathbb{A}_W$-module grading i.e. for each $P_l$, there exists a unique index $j(P_l)$ such that  $P_l \subset \overline{E'}^{j(P_l)}$. By suitable relabeling, we shall assume that if $l_1<l_2$, then $j(P_{l_1})\geq j(P_{l_2})$. 

Now we define $M_l$ to be the $\mathbb{A}_W$-submodule of $\overline{E'}_{\sigma}$ generated by $Z_1, \ldots Z_l$. Denote by $\sigma(Z_l)$ to be one of the $W$-representations $\sigma_1,\ldots, \sigma_p$ to which $P_l$ is isomorphic. Now by considering the natural map from $\mathbb{A}_W \otimes_{\mathbb{C}[W]} Z_i$ to $M_i/M_{i-1}$, we obtain a surjective map from $\mathbf{K}_{\sigma(Z_i)}$ to $M_i/M_{i-1}$. Now we have $\mathrm{dim} \overline{E'}_{\sigma(Z_i)}=\sum_i \mathrm{dim} (M_i/M_{i-1})\leq \mathrm{dim}~ \mathbf K_{\sigma(Z_i)}$. On the other hand, we have $\mathrm{dim} \overline{E'}_{\sigma}=\sum_i\mathrm{dim} E_i=\sum_i\mathrm{dim} \mathbf K_{\sigma_i}$. By using Corollary \ref{cor struct h} and Theorem \ref{thm w struct spr}, we have $\sum_i \mathrm{dim}~\mathbf K_{\sigma_i}=\sum_k \mathrm{dim}~ \mathbf K_{\sigma(Z_k)}$. This shows that the inequality concerning dimensions above has to be an equality. Thus $\mathrm{dim}~ (M_l/M_{l-1})=  \mathrm{dim}~ \mathbf K_{\sigma(Z_l)}$ and so $M_l/M_{l-1} \cong \mathbf K_{\sigma(Z_l)}$. This completes the proof.


\end{proof}

\section{Vanishing Theorem}

We keep using notations in Section \ref{s struct temp mod}. In particular, $\mathbb{H}$ is a graded Hecke algebra of geometric parameters.

\subsection{Twisted-elliptic tempered case}

\begin{definition} 
Recall that $G$ is a complex connected reductive group. The Dynkin diagram automorphism arising from $-w_0$ induces an automorphism, still denoted $\theta$, on $G$. Let $G^{\#}=G \rtimes \langle \theta \rangle$. For $g \in G^{\#}$, let $Z_G(g)$ be the set of elements in $G$ centralizing $g$. Following \cite{CH}, a nilpotent element $e$ is {\it $\theta$-quasidistinguished} if there is no semisimple element $t \in G$ such that $Z_G(t\theta)$ is semisimple and $e$ is distinguished in $Z_G(t\theta)$. Here distinguishedness is defined as \cite{CM}. A tempered module of the form $E_{rh_e,e,r,\zeta}$ (Section \ref{ss deform struct rtm}) is said to be {\it twisted-elliptic} if $e$ is $\theta$-quasidistinguished. 

The terminology is suggested by the study of \cite{Ci, CH}, which establishes connections to twisted elliptic pairings (also see \cite{Ch3}). In \cite[Proposition 3.3]{CH}, it is shown that any $\theta$-quasidistinguished nilpotent element has a solvable centralizer in $\mathfrak{g}$ and vice versa.
\end{definition}

Let $\mathrm{Irr}_{\mathrm{gen}}\widetilde{W}$ be the set of all genuine $\widetilde{W}$-representations. For each $\widetilde{\sigma} \in \mathrm{Irr}_{\mathrm{gen}}\widetilde{W}$, define
\[  a(\widetilde{\sigma})=-\frac{1}{4} \sum_{\alpha>0,\beta>0, s_{\alpha}(\beta)<0}c_{\alpha}c_{\beta}|\alpha||\beta|\frac{\mathrm{tr}_{\widetilde{\sigma}}(\widetilde{s}_{\alpha}\widetilde{s}_{\beta})}{\mathrm{dim} \widetilde{\sigma}}  .
\]

Recall that $V'=\mathbb{C}\otimes_{\mathbb Z}R$. Define $V'{}^{\bot}=\left\{ v \in V: \langle v, v'\rangle=0 \mbox{ for all $v' \in V'$ } \right\}$. We have a $W$-invariant orthogonal decomposition $V=V'\oplus V'{}^{\bot}$ and let $\mathrm{pr}_{V'}:V\rightarrow V'$ be the corresponding projection map. For $v_1,v_2 \in V$, define $\langle v_1,v_2 \rangle_{V'}=\langle \mathrm{pr}_{V'}(v_1),\mathrm{pr}_{V'}(v_2) \rangle$.

\begin{theorem} \cite{Ci, BCT} \label{thm ellip dir coh nonzero}
Let $E$ be a twisted-elliptic tempered module and let $W(rh) \times r$ be the central character of $E$. Let $\mathrm{Irr}_E \widetilde{W}$ be the collection of irreducible $\widetilde{W}$-representation $\widetilde{\sigma}$ such that $a(\widetilde{\sigma})=\langle h, h\rangle_{V'}$. Then 
\begin{enumerate}
\item $H_D(E) \neq 0$;
\item As $\widetilde{W}$-representations, any irreducible summand of $H_D(E)$ lies in $\mathrm{Irr}_E\widetilde{W}$. 
\end{enumerate}
\end{theorem}

\begin{proof}
We remark that the notion of temperedness in \cite{BCT} is different. In order to apply the result, we need the fact that for the Iwahori-Matsumoto dual of a tempered module is 'tempered' in the sense of \cite{BCT}. 

By suitably rescaling on $V$, we first reduce our statement to the case that $E=E_{h_e,e,1,\zeta}$ for some $\theta$-quasidistingusihed $e$. Now using the argument in the proof of Lemma \ref{lem temp z2 grade}, we have that $IM(E_{h_e,e,1,\zeta})$ is unitary and so the hypothesis of \cite[Proposition 5.7]{BCT} is satisfied.  We now consider (1). For each nilpotent element $e$ such that $e$ has a solvable centralizer, using  \cite[Theorem 5.1]{BCT} (or \cite[Theorem 1.0.1]{Ci}) and \cite[Section 3.10]{Ci}, we have a $\widetilde{W}$-type $\widetilde{\sigma}$ such that 
\begin{align} \label{eqn w tilde type}  \mathrm{Hom}_{\widetilde{W}}(\widetilde{\sigma}, IM(E_{rh_e,e,r,\zeta}) \otimes S) \neq 0 .
\end{align}
Now by \cite[Proposition 3.3]{CH}, we can replace the nilpotent $e$ in (\ref{eqn w tilde type}) with $\theta$-quasidistinguished $e$. Then one can apply \cite[Proposition 5.7]{BCT} (or its proof) and Lemma \ref{lem iwahoi mat dirac} to obtain (1). For (2), it follows from the formula of $D^2$ \cite[Theorem 3.5]{BCT}.
\end{proof}


\subsection{Non-twisted-elliptic tempered case}

The non-twisted-elliptic tempered case follows from Vogan's conjecture which is proved in \cite[Theorems 4.4 and 5.8]{BCT}  and says that the Dirac cohomology of an $\mathbb{H}$-module determines its central character if it is nonzero. Hence, we have the following:

\begin{theorem} \label{cor nonellip dir coh zero} \cite[Theorem 5.8]{BCT}
Let $E_{rh_e, e,r,\zeta}$ be a tempered $\mathbb{H}$-module as above. If $e$ is not $\theta$-quasidistinguished (equivalently $E_{rh_e, e,r,\zeta}$ is not twisted-elliptic), then $H_D(E_{rh_e,e,r,\zeta})=0$.
\end{theorem}

\begin{proof}
It is equivalent to show that if $H_D(IM(E_{rh_e,e,r,\zeta}))\neq 0$, then $e$ is $\theta$-quasidistinguished. This follows from \cite[Theorem 5.8]{BCT} and \cite[Proposition 3.3]{CH}. Notice that the argument in \cite[Theorem 5.8]{BCT} still applies to general geometric parameter case by using a surjectivity property in \cite[Section 3.10]{Ci}.
\end{proof}



\subsection{General case}

We now prove the main result in this paper.

\begin{theorem} \label{thm vanihsing dirac coh}
Let $\mathbb{H}$ be a graded Hecke algebra of geometric parameters (Section \ref{ss geo gha}). Let $E$ be a standard module (Definition \ref{def standard mod}). Then $H_D(E) =0$ if and only if $E$ is not twisted-elliptic-tempered .
\end{theorem}

\begin{proof}
We have shown the statement for the case that $E$ is tempered in Theorem \ref{thm ellip dir coh nonzero} and Theorem \ref{cor nonellip dir coh zero}. From now on, we assume that $E$ is not tempered. Write $E=E_{s,e,r,\zeta}$ (Definition \ref{def standard mod}). Define $E'=E\oplus \theta(E)$ if $E\neq \theta(E)$ and let $E'=E$ if $E=\theta(E)$ as in Lemma \ref{lem top generate}. Now $E$ possess a $\mathbb{Z}_2$-grading. 


Let $\sigma$ be a lowest $W$-type for $E$. Let $U_{\sigma}$ be a subspace of $E$ such that $U_{\sigma} \subset E^{\pm}$ and $U_{\sigma} \cong \sigma$ as $W$-representations. Identify $U_{\sigma}$ and $\sigma$. By Lemma \ref{lem filt kato mod}, there exists a filtration $\left\{ \overline{Y}_i \right\}_{i=0}^k$ on $\overline{E'}_{\sigma}$ such that 
\[ 0=\overline{Y}_0 \subset \overline{Y}_1 \subset \overline{Y}_2 \subset \ldots \subset \overline{Y}_k=\overline{E'}_{\sigma} \]
and for each $i=0,\ldots, k-1$, $\overline{Y}_{i+1}/\overline{Y}_i$ is isomorphic to a module $\mathbf{K}_{\sigma_i}$ for some lowest $W$-type $\sigma_i$ of $E'$. 

Let $\mathcal R(\widetilde{W})_{h_e}$ (possibly zero) be the set containing $\widetilde{W}$-representations whose summands are in $\mathrm{Irr}_{h_e}\widetilde{W}$, where $\mathrm{Irr}_{h_e}~\widetilde{W}=\left\{ \widetilde{\sigma} \in \mathrm{Irr}~\widetilde{W}: a(\widetilde{\sigma})=\frac{1}{4}\langle h_e,h_e\rangle_{V'} \right\}$. \\

{\it Claim:} $H_{D_{\mathbb{A}}}(\overline{Y}_{i+1}/\overline{Y}_i)\in \mathcal R(\widetilde{W})_{h_e}$. {\it Proof:}  Set $E_T=E_{rh_e,e,r}$ to be the tempered module in Section \ref{def temp mod}. Since $E_T$ is unitary, the formula of $D^2$ and Lemma \ref{lem stand admit cc} imply that a $\widetilde{W}$-representation $\widetilde{\sigma}$ is a summand of $\mathrm{ker}\pi_{E_T}(D)^2=\mathrm{ker} \pi_{E_T}(D)$ only if $\frac{1}{4}\langle h_e,h_e \rangle_{V'} = a(\widetilde{\sigma})$ (see the paragraph before \cite[Section 4.1]{BCT}). Here $\pi_{E_T}$ defines the action of $\mathbb{H}$ on $E_T$. Hence we have $H_D(E_T)=\mathrm{ker} D \in \mathcal R(\widetilde{W})_{h_e}$. Now Proposition \ref{prop unitary dirac} and Corollary \ref{cor struct h} imply that $H_D(\overline{Y}_{i+1}/\overline{Y}_i)\in \mathcal R(\widetilde{W})_{h_e}$. 

{\it Claim:} $H_{D_{\mathbb{A}}}(\overline{E}_{\sigma})\in \mathcal R(\widetilde{W})_{h_e}$. {\it Proof:} This follows from a version of the six-term short exact sequence (Proposition \ref{prop six term es}) for $\mathbb{A}_W$-modules and the previous claim. \\

We are now ready to show that $H_D(E)=0$. Decompose, as $\widetilde{W}$-representations, $\overline{E}_{\sigma}\otimes S=\widetilde{X}_1 \oplus \widetilde{X}_2$ such that $\widetilde{X}_1 \in\mathcal R(\widetilde{W})_{h_e}$ and $\widetilde{X}_2$ has no summands in  $R(\widetilde{W})_{h_e}$. Then we can apply Corollary \ref{cor dirac cohom def1}. With the previous claim, we have $H_D(E) \in \mathcal R(\widetilde{W})_{h_e}$. Now take $x \in H_D(E)$. The formula of $D^2$ implies that $(-\frac{1}{4}\langle s,s \rangle_{V'}+r^2a(\widetilde{\sigma}))x=0$ for some $\widetilde{\sigma}\in \mathrm{Irr}\widetilde{W} \cap R(\widetilde{W})_{h_e}$. Hence we have $\frac{1}{4}(-\langle s, s\rangle_{V'}+r^2\langle h_e,h_e \rangle_{V'})x=0$. Since we are considering $E$ is not tempered, we have $\langle s, s \rangle_{V'} \neq r^2\langle h_e, h_e\rangle_{V'}$ (note that $\langle s-rh_e,rh_e\rangle_{V'}=0$). Thus $x=0$ and so $H_D(E)=0$. This completes the proof.
\end{proof}

\subsection{Twisted Dirac index} \label{ss dirac index}
Our formulation of Dirac index is a slight variation of \cite{CH}. Let $X=X^+\oplus X^-$ be a $\mathbb{Z}_2$-graded $\mathbb{H}$-module admitting a central character. 

Define $\mathcal S=S$ when $\mathrm{dim} V'$ is even. Define $\mathcal S=S^+\oplus S^-$, where $S^+$ and $S^-$ are two inequivalent simple $C(V')$-modules.
Define 
\[  I(X)=(X^+-X^-)\otimes \mathcal S .
\]
Here we regard $I(X)$ as a virtual $\widetilde{W}$-representation with the action given from the diagonal embedding $\Delta$ (see \ref{eqn diagonal embed}).

We similarly define as in Section  \ref{ss six term} that $\pi_{\mathcal S}(D)^{\pm}=\pi_{\mathcal S}(D)|_{X^{\pm}\otimes \mathcal S}$, where $\pi_{\mathcal S}: \mathbb{H} \otimes C(V') \rightarrow \mathrm{End}_{\mathbb{C}}(X\otimes S)$ is the map defining the action of $\mathbb{H} \otimes C(V')$ on $X\otimes S$. Define
\[ \mathcal H^{\pm}_D(X) =\mathrm{ker}~ \pi_{\mathcal S}(D)^{\pm}/(\mathrm{ker}~ \pi_{\mathcal S}(D)^{\pm} \cap \mathrm{im}~ \pi_{\mathcal S}(D)^{\mp}) .
\]

\begin{lemma}
As virtual $\widetilde{W}$-representations,
\[  I(X) =\mathcal H_D^+(X)-\mathcal H^-_D(X) .\]
\end{lemma}
\begin{proof}
Decompose $X \otimes S$ into $\lambda$-eigenspaces $\widetilde{X}_{\lambda}$ of $\pi_{\mathcal S}(D)^2$. By linear algebra considerations, $I(X)$ is reduced to the expression $\widetilde{X}^+_0-\widetilde{X}^-_0$, where $\widetilde{X}^{\pm}_0=(X^{\pm}\otimes S)\cap \widetilde{X}_0$. Now the expression in right hand side follows from definitions.
\end{proof}

As a consequence, we have:
\begin{corollary} \label{cor dirac index}
Let $E$ be a standard module. Define $E'$ as in Lemma \ref{lem top generate}. Then $I(E')=0$ if and only if $E$ is not twisted-elliptic tempered.
\end{corollary}

An alternate approach to the above corollary is to use a twisted Euler-Poincar\'e pairing in \cite{Ch3} (also see \cite{CT}).

\section{Dirac cohomology of ladder representations} \label{s dirac ladder}
Ladder representations for $p$-adic groups are introduced by Lapid-M\'inguez \cite{LM} in terms of the Zelevinsky classification. We shall work on the graded Hecke algebra setting (see \cite{BC2}). Fix integers $l$ and $n$ throughout the whole section.

\subsection{Zelevinsky segments and ladder representations} \label{ss zel and ladd}

 Let $\mathfrak{h}_l$ be the set of diagonal matrices in $\mathfrak{gl}(l,\mathbb{C})$. Let $\mathcal O$ be the zero orbit and let $\mathcal L$ be the trivial local system on $\mathcal O$. The datum $(\mathfrak{h}, \mathcal O, \mathcal L)$ is a cuspidal triple. This gives a graded Hecke algebra of type $A_{l-1}$. We shall do some normalizations below.

Let $S_l$ be the symmetric group permuting on $l$ numbers. We also set $W=S_l$. Let $s_{i,i+1} \in S_l$ be the transposition switching $i$ and $i+1$. Let $\epsilon_1,\ldots, \epsilon_l$ be a standard basis for $\mathfrak{h}_l$. Let $R_l^+=\left\{ \epsilon_i-\epsilon_j: i<j \right\}$ be a fixed set of positive roots. Let $\Pi_l$ be the set of simple roots in $R^+_l$. Set $c \equiv 1$. Define bilinear form $\langle , \rangle$ on $\mathfrak{h}_l$ satisfying $\langle \epsilon_i,\epsilon_j \rangle=\delta_{i,j}$. Let $\mathbb{H}_l=\mathbb{H}(\mathfrak{h}_l, S_l, \Pi_l, c,\mathbf r)$ (Definition \ref{def gaha}). We remark that our definition of $\mathbb{H}_l$ differs from the one \cite{AS} by a sign in a commutation relation.

\begin{definition} \label{def segment}
A pair $\Delta=[a,b]$ of complex numbers is called a {\it segment} if $b\geq a$. Let $l(\Delta)=b-a+1$, which will be called the length of $\Delta$. A {\it multisegment} $\mathfrak{m}=\left\{ \Delta_1, \ldots, \Delta_n \right\}$ is a collection of ordered (possibly repeated) segments $\Delta$ satisfying $\sum_{i=1}^n l(\Delta_i)=l$. The interesting cases, especially for non-zero Dirac cohomology, are for $a,b$ being integers in a segment $\Delta=[a,b]$. From now on, we shall also assume that $a$ and $b$ for a segment $[a,b]$ are integers. For $e \in \mathbb Z$ and a multisegment $\mathfrak{m}=\left\{ [a_1,b_1], [a_2,b_2],\ldots, [a_n,b_n] \right\}$, define $m(\mathfrak{m},e)=\mathrm{card}\left\{ i : a_i \leq e  \leq b_i \right\}$.  For example, if $\mathfrak{m}=\left\{ [4,5], [2,4],[1,3] \right\}$ then $m(\mathfrak{m},1)=1$, $m(\mathfrak{m},2)=2$, $m(\mathfrak{m},3)=2$, $m(\mathfrak{m},4)=2$, $m(\mathfrak{m},5)=1$. 
\end{definition}

Let $\mathfrak{m}=\left\{ \Delta_1,\ldots, \Delta_n\right\}$ be a multisegment. Let $l_i=l(\Delta_i)$ for each $i$. Identify the group $S_{l_1}\times \ldots \times S_{l_n}$ to be a natural subgroup of $S_l$ via the map $1 \otimes \ldots \otimes 1\otimes s_{k,k+1} \otimes 1 \otimes \ldots \otimes 1 \mapsto s_{A_i+k, A_i+k+1}$, where $s_{k,k+1}$ is in the $k$-th factor of LHS and $A_i=l_1+\ldots +l_{i-1}$. Let $\mathbb{H}_{l_1}\otimes \ldots \mathbb{H}_{l_n}$ be the subalgebra generated by $v \in V$ and $t_{w}$ for $w \in S_{l_1}\times \ldots \times S_{l_r}$. Define a $1$-dimensional $\mathbb{H}_{l_1}\otimes \ldots \mathbb{H}_{l_n}$-module $\mathbf{1}_{\mathfrak{m}}=\mathbb{C}x$ determined by
\[   \epsilon_k. x = (a_i+k-A_i-1)x  \quad \mbox{ for $A_i < k \leq A_{i+1}$}
\]
\[   t_w.x =(-1)^{l(w)} x
  \]

Define 
\[ E(\mathfrak m)=\mathbb{H}_l \otimes_{\mathbb{H}_{l_1}\otimes \ldots \otimes \mathbb{H}_{l_n}} \mathbf{1}_{\mathfrak m}  .\]

Let $\mathcal Z_l$ be the collection of all Zelevinsky segments $\mathfrak{m}=\left\{ [a_1,b_1],\ldots, [a_n,b_n] \right\}$ with the property that $b_1 > b_2 >\ldots > b_n$. For $\mathfrak{m} \in \mathcal Z_l$, $E(\mathfrak{m})$ has a unique simple quotient denoted by $L(\mathfrak{m})$, which is a special case of Langlands-Zelevinsky classification. (This can be checked, for example, by results of \cite{Ev}.) This coincides with a standard module in Section \ref{ss deform struct rtm} (see \cite[Section 3]{Lu3}), up to the Iwahori-Matsumoto dual. In view of Lemma \ref{lem iwahoi mat dirac}, this will not play a significant role in our arguments and will be implicitly used in the rest of this section.

\begin{proposition} \label{prop ellip segment}
An $\mathbb{H}_l$-module $X$ is twisted-elliptic tempered if and only if $X=E(\mathfrak{m})$ for some $\mathfrak{m} \in \mathcal Z_l$ satisfying that $\mathfrak{m}=\left\{[-b_1,b_1],\ldots,[-b_n,b_n] \right\}$.
\end{proposition}

Let $\mathcal Z_l^{\mathrm{ellip}}$ be a subset of $\mathcal Z_l$ consisting of $\mathfrak{m}$ satisfying the property: there exists $\mathfrak{m}'$ of the form $\left\{[-b_1,b_1],\ldots, [-b_n,b_n] \right\}$ such that $m(\mathfrak{m},e)=m(\mathfrak{m}',e)$ for all $e \in \mathbb{Z}$. 

Note that for each $\mathfrak{m} \in \mathcal Z_l^{\mathrm{ellip}}$, there is exactly one element, denoted $\mathrm{m}^{\mathrm{temp}}$, in $\mathcal Z_l^{\mathrm{ellip}}$ such that $E(\mathfrak{m}^{\mathrm{temp}})=L(\mathfrak{m}^{\mathrm{temp}})$ is the twisted-elliptic tempered module supported at the same central character of $L(\mathfrak{m})$. From Proposition \ref{prop ellip segment}, we have:

\begin{corollary}
For any $\mathfrak{m} \in \mathcal Z_l$, $L(\mathfrak{m})$ has a central character of a twisted-elliptic tempered module if and only if $\mathfrak{m} \in \mathcal Z_l^{\mathrm{ellip}}$. 
\end{corollary}

\begin{definition}
Define $\mathcal Z_l^{\mathrm{ladd}}$ to be the collection of elements $\mathfrak{m}=\left\{[a_1,b_1],\ldots, [a_n,b_n] \right\}$ in $\mathcal Z_l$ such that $a_1>a_2>\ldots >a_n$. For any $\mathfrak{m} \in \mathcal Z^{\mathrm{ladd}}_l$, we call $L(\mathfrak{m})$ to be the ladder representation for $\mathbb H_l$.
\end{definition}

\subsection{Arakawa-Suzuki functor} \label{s AS functor}

We keep using notations in the previous subsection. We shall consider $\mathfrak{gl}_n$-modules in the BGG category $\mathscr O$. Let $\mathfrak{n}_-$ be the Lie subalgebra of $\mathfrak{g}$ containing all strictly lower triangular matrices. For $X \in \mathscr O$, define
\[  H_0(\mathfrak{n}_-, X)=X/\mathfrak{n}_-X ,\]
the $0$-th $\mathfrak{n}_-$-homology on $X$.

Let $\mathcal V=\mathbb{C}^n$ be a $\mathfrak{gl}(n,\mathbb{C})$-representation by the matrix multiplication. Let $\mathfrak{t}_n$ be the set of diagonal matrices in $\mathfrak{gl}(n,\mathbb{C})$. We naturally identify $\mathfrak{t}_n$ with $\mathbb{C}^n$. Let $\epsilon_1',\ldots, \epsilon_n'$ be a standard basis for $\mathfrak{t}_n$. Using the standard bilinear form on $\mathbb{C}^n$, we also identify $\mathfrak{t}_n^{\vee}$ with $\mathbb{C}^n$. Similar to $\mathfrak{h}_l$ in the previous section, we have a set $R_n^+$ of positive roots and a set $\Pi_n$ of simple roots. Let $\rho$ be the half sum of all the positive roots in $R_n^+$.

For a $\mathfrak{t}_n$-weight $\mu=(a_1,\ldots, a_n)$, define $L(\mu)$ to be the irreducible module $\mathfrak{sl}(n,\mathbb{C})$ with the highest weight module $\nu$ in the BGG category $\mathscr O$.

\begin{definition}\cite{AS} \label{def AS functor} Define the Arakawa-Suzuki functor  $F_{\gamma}: \mathscr O \rightarrow \mathbb{H}-\mathrm{mod}$ as follows:
\[ F_{\gamma}(X)= H_0(\mathfrak{n}_-, X \otimes \mathcal V^{\otimes l})_{\gamma} .
\]
The $S_l$-action on $F_{\gamma}(X)$ is defined by 
\[ w. (x \otimes v_1\otimes \ldots \otimes v_l)=(-1)^{l(w)}x \otimes v_{w(1)}\otimes \ldots \otimes v_{w(l)} .\]
One extends the $S_l$-structure to $\mathbb{H}_l$-module by using a Casimir element, see \cite{AS}. We shall not need the explicit description of the action for $S(\mathfrak{h}_l)$.
\end{definition}

Let $P(\mathcal V^{\otimes l})$ be the set of $\mathfrak{t}_n$-weights of $\mathcal V^{\otimes l}$.  We suppose that $\gamma-\mu \in P(\mathcal V^{\otimes l})$. Then $\gamma-\mu=\sum_{i=1}^nl_i\epsilon_i$ for non-negative integers $l_1,\ldots, l_n$ satisfying $l_1+l_2+\ldots+l_n=l$. Let $\rho=\sum_{\alpha \in R_l^+} \alpha$. Define
\[\mathfrak{m}_{\gamma,\mu}:=\left\{ [ \mu_1',\mu_1'+l_1-1],\ldots , [\mu_n',\mu_n'+l_n -1] \right\},
\]
where $\mu_i'=\langle \mu +\rho, \epsilon_i \rangle= \langle \mu,\epsilon_i \rangle +\frac{n-2i+1}{2}$. 

Let $D_n^o=\left\{ \gamma : \langle \gamma , \alpha \rangle \geq 1 \mbox{ for all $\alpha \in R^+ $} \right\}$. We recall a result of Suzuki. 

\begin{theorem} \cite[Theorem 3.2.2]{Su} \label{thm suzuki functor}
Let $\langle \gamma+\rho,\alpha \rangle>0$ and let $\mu \in \mathcal V$ such that $\lambda-\mu \in P(\mathcal V^{\otimes l})$. Then
$F_{\gamma}(L(\mu)) \cong L(\mathfrak m_{\gamma, \mu})$
\end{theorem}

For each $\mathfrak{m}=\left\{ [a_1,b_1],\ldots, [a_n,b_n]\right\} \in \mathcal Z_l^{\mathrm{ladd}}$, let $\gamma=(b_1,\ldots, b_n)-\rho=(b_1-\frac{2n-1}{2},\ldots, b_n+\frac{2n-1}{2})$. Let $\mu=(a_1,\ldots, a_2)-\rho$. Then $\mu$ satisfy the condition in Theorem \ref{thm suzuki functor} and $\mathfrak{m}_{\gamma,\mu}=\mathfrak{m}$. The Arakawa-Suzuki functor gives the following BGG resolution \cite{Su}:

\begin{theorem} \cite[Theorem 5.1.1]{Su}
Let $\mathfrak{m}\in \mathcal Z_l^{\mathrm{ladd}}$ with $\gamma$ and $\mu$ as above. There exists a BGG-type resolution for the $\mathbb{H}_l$-module $L(m_{\gamma, \mu})$:
\[  0 \rightarrow \bigoplus_{w \in S_n: l(w)=l(w_0)} E(\mathfrak{m}_{\gamma,\mu}^w) \rightarrow \ldots \rightarrow \bigoplus_{w \in S_n: l(w)=1} E(\mathfrak{m}_{\gamma,\mu}^w) \rightarrow E(\mathfrak{m}_{\gamma,\mu}) \rightarrow L(\mathfrak{m}_{\gamma,\mu}) \rightarrow 0,
\]
where $w_0$ is the longest element in $S_l$, and $\mathfrak{m}^w_{\gamma,\mu}=\left\{ [a_{w(1)},b_1],\ldots, [a_{w(n)}, b_n] \right\}$. (If $a_{w(k)}>b_k+1$, then set $E(\mathfrak{m}^w_{\gamma, \mu})=0$ as convention. ) 
\end{theorem}

\subsection{Dirac cohomology}

We need some notations for representations of $S_l$. A partition $\lambda$ of $l$ is a sequence $(l_1,l_2,\ldots, l_n)$ of integers such that $l_1+\ldots+l_n=l$. For a partition $\lambda$ of $l$, we define $l(\lambda)$ to be the number of parts in $\lambda$. A partition $\lambda$ is said to be even or odd if $l-l(\lambda)$ is even or odd respectively \cite[Section 2]{St}. We denote $\sigma_{\lambda}$ to be the irreducible $S_l$-representation associated to $\lambda$. In particular, if $\lambda=(l)$, then $\sigma_{\lambda}$ is the trivial representation.

Denote by $\widetilde{S}_l$ the spin double cover of $S_l$. We now describe the irreducible genuine representations of $\widetilde{S}_l$ \cite{Sc} (see \cite[Section 7]{St}). Denote by $DP_l$ the set of partitions of $l$ with distinct parts. Denote by $DP_l^+$ (resp. $DP_l^-$) the subset of $DP_l$ consisting of even partitions (resp. odd partitions). Following \cite{St}, one can associate one irreducible projective representations, denoted $\widetilde{\sigma}_{\lambda}$, when $\lambda \in DP_l^+$ and can associate to two irreducible projective representations, denoted $\widetilde{\sigma}_{\lambda}^{\pm}$ when $\lambda\in DP_l^-$. Those $\widetilde{\sigma}_{\lambda}$ ($\lambda \in DP_l^+$) and $\widetilde{\sigma}_{\lambda}^{\pm}$ ($\lambda \in DP_l^-$) are mutually non-isomorphic to each other and form a complete list of irreducible projective $\widetilde{S}_l$-representations. In particular, $S=\widetilde{\sigma}_{(l)}$ if $l$ is odd and $S=\widetilde{\sigma}_{l}^{\pm}$ if $l$ is even,

For each $\mathfrak{m} \in \mathcal Z_l$, denote by $\lambda(\mathfrak m)$ the associated partition i.e. the parts in $\lambda(\mathfrak{m})$ are given by the numbers $b_i-a_i+1$. For example, if $\mathfrak m=\left\{[3,7],[2,6], [1,3] \right\}$, $\lambda(\mathfrak{m})=(5,5,3)$. We shall sometimes write $\mathfrak{m}$ for $\lambda(\mathfrak{m})$ if the context is clear.

\begin{theorem} \label{thm dirac ladder}
Let $\mathfrak{m} \in \mathcal Z_l^{\mathrm{ellip}} \cap \mathcal Z_l^{\mathrm{ladd}}$. Let $\lambda=\lambda(\mathfrak{m}^{\mathrm{temp}})$. Set $\widetilde{\sigma}=\widetilde{\sigma}_{\lambda}$.
\begin{enumerate}
\item[(a)] If $\lambda \in DP_l^+$ and set $\widetilde{\sigma}=\widetilde{\sigma}_{\lambda}^+\oplus \widetilde{\sigma}_{\lambda}^-$.  Then if $\lambda \neq (n)$,
\[   H_D(L(\mathfrak{m})) =\bigoplus_{i=1}^k \widetilde{\sigma},
\]
where 
\[  k=\frac{1}{\varepsilon_{\lambda} \varepsilon_{(n)}} 2^{(l(\lambda)-1)/2}
\]
Here $\varepsilon_{\lambda'}=\sqrt{2}$ if $\lambda' \in DP_l^-$ and $\varepsilon_{\lambda'}=1$ if $\lambda' \in DP_l^+$. 
\item[(b)] If $\lambda=(n)$, $H_D(L(\mathfrak{m}))=S$.
\end{enumerate}
\end{theorem}
\begin{proof}
There is only one $w \in S_n$ such that $\mathfrak{m}^w=\mathfrak{m}^{\mathrm{temp}}$. For any $\mathfrak{m} \in \left\{ \mathfrak{m}^w \neq 0: w\in S_n \right\} \setminus \left\{ \mathfrak{m}^{\mathrm{temp}} \right\}$, $E(\mathfrak{m})$ is not a twisted-elliptic tempered and hence $H_D(E(\mathfrak{m}))=0$ by Theorem \ref{thm vanihsing dirac coh}. Thus by applying the six-term exact sequence (Proposition \ref{prop six term es}) repeatedly, we have 
\[ H_D(L(\mathfrak{m})) \cong H_D(E(\mathfrak{m}^{\mathrm{temp}})) .\]
$H_D(E(\mathfrak{m}^{\mathrm{temp}}))$ is computed in \cite[Lemma 3.6.2]{BC} from \cite[Theorem 9.3]{St}.
\end{proof}

\subsection{}

Two segments $[a,b]$ and $[a',b']$ are said to be {\it linked} if $b+1=a'$ or $b'+1=a$. We say that $[a,b]$ is left (resp. right) linked to $[a',b']$ if $b+1=a'$ (resp. $b'+1=a$). 

The following lemma immediately follows from the definition of $\mathcal Z_l^{\mathrm{ladd}}$.
\begin{lemma}
Let $\mathfrak{m} \in \mathcal Z^{\mathrm{ladd}}_l$. For each segment $[a,b]$ in $\mathfrak{m}$, there is at most one segment in $\mathfrak{m}$ which is left linked to $[a,b]$. The same is also true by replacing left with right.
\end{lemma}


\begin{definition} 
We say that a multisegment $\mathfrak m$ has {\it up-and-then-down property} if there exists an integer $N$ such that for any $i<j\leq N$, $m(\mathfrak m,i) \leq m(\mathfrak m,j)$ and for any $N \leq i<j$, $m(\mathfrak m,i)\geq m(\mathfrak m,j)$.   
\end{definition}

Following from definitions, we immediately have the following:
\begin{lemma} \label{lem updown}
For any $\mathfrak{m} \in \mathcal Z_l^{\mathrm{ellip}}\cap \mathcal Z_l^{\mathrm{ladd}}$, $\mathfrak{m}$ has up-and-then-down property.
\end{lemma}



\noindent
\subsection{ Algorithm for $w(\mathfrak{m})$.} \label{ss alg for w}

\begin{definition} \label{def set Fm}
Let $\mathfrak{m}=\left\{ [a_1,b_1],\ldots, [a_n,b_n] \right\} \in \mathcal Z_l^{\mathrm{ladd}}$. Define $\sim$ as follows: $[a_i,b_i] \sim [a_j,b_j]$ if and only if $[a_i,b_i]$ and $[a_j,b_j]$ are linked. Let $\mathcal F(\mathfrak{m})$ be the set of equivalence classes of the relation generated by $\sim$.

For each  $f \in \mathcal F(\mathfrak{m})$, we write all the elements in the form $[a^1,b^1],\ldots, [a^p,b^p]$, where  $p=\mathrm{card}(f)$ and after suitable relabeling, we may assume that $[a^i,a^{i+1}]$ are right linked to $[a^{i+1},b^{i+1}]$ for each $i$. We let $J(f)=[a^p,b^1]$. For such $f$, define $a(f)=a^p$ and $b(f)=b^1$.
\end{definition}

\noindent
Algorithm: Let $\mathfrak{m} \in \mathcal{Z}^{\mathrm{ladd}}_l$. Suppose $\mathrm{card}~\mathcal F(\mathfrak{m})=s$.
\begin{enumerate}
\item[Step 1] We order the set $\mathcal F(\mathfrak{m})$ in two ways: $\left\{ f_{i_1}, f_{i_2},\ldots, f_{i_s} \right\}$ and $\left\{ f_{j_1}, \ldots, f_{j_s} \right\}$ such that $b(f_{i_1}) >b(f_{i_2})> \ldots > b(f_{i_r})$ and $a(f_{j_1})< a(f_{j_2})<\ldots <a(f_{j_r})$. 
\item[Step 2] For each $f_k$, we write the segments in $f_k$: 
\[ f_k=\left\{ [a(k,1),b(k,1)],[a(k,2),b(k,2)],\ldots, [a(k,p_k),b(k,p_k)] \right\},\] where $p_k$ is the number of segments in $f_k$ and those $a$ and $b$ satisfy the relation that $b(k,e)+1=a(k,e-1)$ for $e=2,\ldots, p_k$. 
\item[Step 3] Now define a function $G: \left\{ a_1,\ldots, a_{n} \right\} \rightarrow \left\{ a_1, \ldots, a_n \right\}$ such that
\[  G(a(i_e,d))=\left\{  \begin{array}{ll} a(j_e, p_{j_e}) & \mbox{ if $d=1$} \\  a(j_e,d-1) & \mbox{ if $d>1$ } \end{array} \right. 
\]   
\item[Step 4] We now define an element $w(\mathfrak{m}) \in S_n$ such that $G(a_i)=a_{w(\mathfrak{m})(i)}$. 
\end{enumerate}

\begin{example}
$\mathfrak{m}=\left\{[7,10],[4,8],[3,6]\right\}$, we have $w(\mathfrak{m})=(1,3) \in S_3$. \\
\end{example}

\begin{example}
Let $\mathfrak{m}=\left\{[5,7],[3,5],[2,4],[1,3] \right\}$. Then $J(f_{i_1})=[2,7]$, $J(f_{i_2})=[3,5]$, $J(f_{i_3})=[1,3]$, $J(f_{j_1})=[1,3]$, $J(f_{j_2})=[2,7]$, $J(f_{j_3})=[3,5]$. Then we have $w(\mathfrak{m})=(1,4,2,3) $.
\end{example}


\noindent
\subsection{ Partition $\alpha(\mathfrak{m})$} \label{ss partition alpha}
Let $\mathfrak{m} \in \mathcal Z_l^{\mathrm{ladd}} \cap \mathcal Z_l^{\mathrm{ellip}}$. We repeat Steps 1 and 2 in Section \ref{ss alg for w} to obtain the indices $f_{i_k}$, $f_{j_k}$, $a(k,e)$ and $b(k,e)$. 
\begin{definition}
For a partition $\alpha=(l_1,\ldots, l_n)$ of $l$, let $\alpha^T=(t_1,\ldots, t_r)$, where $\alpha^T$ is the transpose of $\alpha$. 

We define the length of the $e$-th hook of $\alpha$ to be the number $l_e+t_e-2e+1$. Hence the length of the $e$-th hook is to count the number of boxes starting from the rightmost of the $e$-th row to the diagonal box, and then from the diagonal box down to the bottom of the $e$-th row. For example, for a hook partition $\alpha=(5,1,1,1)$. The length of the first hook is 8.

 We define the height of the $e$-th hook of $\alpha$ to be the number $t_e-e+1$. 
\end{definition}

For $e=1,\ldots,s$, define $\mathrm{hk}(\mathfrak{m}, e)=b(i_e,1)-a(j_e,p_{j_e})+1=b(f_{i_e})-a(f_{j_e})+1$. Define 
\[ \mathrm{ht}(\mathfrak{m}, e)=\mathrm{card}\left\{ k:  a(j_e,p_e) \leq b_k \leq b(i_e,1)  \right\} .\] 
Define a partition $\alpha'(\mathfrak m)$ associated to $\mathfrak{m}$ such that $\mathrm{hk}(\mathfrak m, e)$ gives the length of the $e$-th
hook and $\mathrm{ht}(\mathfrak m, e)$ gives the heights of hooks. The definition of $\mathcal Z_l^{\mathrm{ladd}}$ assures that $\alpha'(\mathfrak m)$ is
well-defined i.e. a sequence of non-increasing positive integers. Define $\alpha(\mathfrak{m})=\alpha'(\mathfrak{m})^T$ (the transpose of $\alpha'(\mathfrak{m})$).


\begin{lemma} \label{lem count ht}
Let $e=1,\ldots, s$. Let $w=w(\mathfrak{m})$ defined in Section \ref{ss alg for w}. Let $N$ be the integer such that $a_N=a(i_e,1)$. Then
\[ \mathrm{ht}(\mathfrak{m}, e) =w(N)-N+e .  \]
\end{lemma}

\begin{proof}
In the notation of Section \ref{ss alg for w}, we have a corresponding function $G$ which permutes $a_1,\ldots, a_l$. Then $G(a_N)=G(a(i_e,1))=a(j_e, p_e)$. Now we have $|w(N)-N|+1$ counts the number of elements in 
\[  \left\{ i:    b^* \leq   b_i  \leq b' \right\}
\]
where $b^*=b_{w(N)}$ and $b'=b_N$ if $w(N)<N$; $b^*=b_{w(N)}$ and $b'=b_{w(N)}$ otherwise. Note that we have $\mathrm{card}\left\{ i: a_{w(N)} \leq b_i < b_N \right\}$ is $e-1$ (by the up-and-then-down property). Now by considering two separate cases, we have the statement.
\end{proof}

\begin{lemma} \label{lem counting ht}
Let $\mathcal A(e)=\left\{ a_s:  a(j_e,p_e) \leq b_s \leq b(i_e,1) , a_s \neq a(t,1) \ \mbox{ for some $t$ } \right\}$. For any $a_i \in \mathcal A(e) \setminus \mathcal A(e+1)$, there is $a_{N}$ such that $a_N=b_i+1$ and furthermore, $a_{i-k}=a_N$.
\end{lemma}
\begin{proof} 
Let $b_i=a(i_k,q_{k})=a(j_{k'},q_{k'}) \in \mathcal A(e)$ for some $q_k$ and $q_{k'}$. We now look at $f_{i_{k}}$ and $f_{j_{k'}}$ in $\mathcal F(\mathfrak m)$. The  definitions of $\mathcal A(k)$ (which does not contain the rightmost end $a(j_k,1)$) guarantees that there are $a_{N}=b_i$. Now the property of $\mathcal Z_l^{\mathrm{ladd}}$ and up-and-then-down property guarantee that there are exactly $k-1$ $a_t$'s such that $a_i<a_t <a_N=b_i$. 
\end{proof}

\begin{lemma} \label{lem bij two counting sets}
Keep using the notation in the previous lemma. Let $|\mathcal F(\mathfrak m)|=s$. For $1 \leq k \leq s$, we have
\[  |\mathcal A(e) \setminus \mathcal A(e+1)|=\mathrm{ht}(\mathfrak{m},e)-\mathrm{ht}(\mathfrak{m},e+1)-1
\]
(Here $\mathcal A(s+1)=\emptyset$.)
\end{lemma}

\begin{proof}
By definitions, we have $\mathcal A(e)=\mathrm{ht}(\mathfrak{m}, e)-(s-e+1)$. Here the term $s-e+1$ comes from the strict equality for $a(j_e,1)>b_i$ in the definition of $\mathcal A(e)$. Hence we have $|\mathcal A( e) \setminus \mathcal A( e+1)|=\mathrm{ht}(\mathfrak{m},e)-\mathrm{ht}(\mathfrak{m}, e+1)-1$.
\end{proof}

\subsection{Existence of a $W$-representation}

\begin{proposition} \label{prop existence can w}
Let $\mathfrak{m} \in \mathcal Z^{\mathrm{ladd}}_l \cap \mathcal Z_l^{\mathrm{ellip}}$. Recall that $\alpha(\mathfrak{m})$ is defined in Section \ref{ss partition alpha}. Set $\alpha=\alpha(\mathfrak{m})$. Then
\[  \mathrm{Hom}_{W}(\sigma_{\alpha}, L(\mathfrak{m})|_{W}) \neq 0 .
\]
\end{proposition}

\begin{proof}
Construct $\alpha(\mathfrak m)$ as in Section \ref{ss partition alpha}. Write $\alpha(\mathfrak{m})=(m_1,\ldots, m_n)$ as a partition of $n$. Let $\omega=(m_1,\ldots, m_n)$ as a $\mathfrak{t}_n$-weight. By the classical Schur-Weyl duality, there is a subrepresentation in $\mathcal V^{\otimes l}$, which is isomorphic to $\alpha(\mathfrak{m}) \boxtimes L(\omega)$ as $\mathbb{C}[S_l] \otimes \mathfrak{gl}(n,\mathbb{C})$-representation. (Here the $S_l$-action on $\mathcal V^{\otimes l}$ is defined as in Definition \ref{def AS functor}.) Recall from Section \ref{s AS functor}  that there exist $\mathfrak{t}_n$-weights $\gamma, \mu$ such that $\mathfrak{m}=\mathfrak{m}_{\gamma, \mu}$, and we have $L(\mathfrak{m}) \cong F_{\gamma}(L(\mu))$ in Theorem \ref{thm suzuki functor}. Thus to prove the proposition, it suffices to show that $(L(\mu) \otimes L(\omega))_{\gamma}$ is non-zero. This shall follow from the following claim and the PRV conjecture, proved by Kumar \cite{Ku}.\\

\noindent
{\it Claim:} There exists $w_1 \in S_n$ and $w_2 \in S_n$ such that $w_1\mu+w_2\omega=\gamma$. \\

\noindent
{\it Proof of the claim:} We choose $w_1=w(\mathfrak{m})$ and set $w:=w_1$ for simplicity. Let $\mathfrak{m}=\left\{[a_1,b_1],\ldots,[a_n,b_n] \right\}$ with $a_1>a_2>\ldots>a_n$. Let $\mu'=(a_1,\ldots, a_n)$ and let $\gamma'=(b_1,\ldots, b_n)$. Recall that from the previous construction we have $\mu=\mu'-\rho$ and $\gamma=\gamma'-\rho+1$. We also write $(a_1',\ldots, a_n')=\mu$ and $(b_1',\ldots, b_n')=\gamma$. We also construct the set $\mathcal F(\mathfrak{m})=\left\{ f_1,\ldots, f_s \right\}$ as in Definition \ref{def set Fm}.

To prove the claim, it is equivalent to show that $\gamma-w_1\mu \in S_n\omega$. To this end, it suffices to show for each $m_i$, there exists $k$ such that $b_k'-a_{w(k)}'=m_i$ (with the same multiplicity).

We first compute those $m_t$. For $1 \leq t \leq s$, we have 
\[ m_t =\mathrm{hk}(\mathfrak{m},t)-\mathrm{ht}(\mathfrak{m},t)+e . \]
We now consider $t >s$. For $\mathrm{ht}(s)+s-1 \geq t >s$, we have $m_t=s$. In general, for $\mathrm{ht}(\mathfrak{m},s-x)+(s-x)-1\geq t >\mathrm{ht}(\mathfrak{m},s-x+1)+(s-x)$, we have $m_t=s-x$.

 We consider the tableaux obtained from $\gamma - \widetilde{\mu}_w$. For the first row, since $w(1)=w_1(1)=n$, we have
\[ b_1'-a_{w(1)}' = b_1-\frac{n-1}{2}+1-(a_{w(1)}-\frac{n-2w(1)+1}{2}) =b_1-a_n+1-(n-1) =\mathrm{hk}(\mathfrak{m},1)-\mathrm{ht}(\mathfrak{m},1)+1
\]
 For $ 1 \leq t \leq s$, we choose $N$ such that $b_N=b(i_t,1)=b(f_{i_t})$. By definitions,  $a_{w(N)}=a(j_t,p_{j_t})=a(f_{j_t})$. (Note that when $t=1$, $i_t=1$ and $j_t=n$ and the that case has been treated.)
\begin{align*}
 b_N'-a_{w(N)}' =& b_{N}+1-\frac{n-2N+1}{2}-(a_{w(N)}-\frac{n-2w(N)+1}{2}) \\
                =& b(f_{t_j})+1-a(j_t,p_{j_t})-\frac{2w(N)-2N}{2} \\
                =& \mathrm{hk}(\mathfrak{m},t) -\mathrm{ht}(\mathfrak{m},t)+e \quad \mbox{ by Lemma \ref{lem count ht}}
\end{align*}
We now consider $t >s$. We consider the case that $\mathrm{ht}(\mathfrak{m},s)+s-1 \geq t>s$. By Lemma \ref{lem bij two counting sets}
\[ \mathrm{card} \left\{ t: \mathrm{ht}(\mathfrak{m}, s)+s-1 \geq t >s  \right\} = \mathrm{card}~ \mathcal A(s) .\]
For any  $\mathrm{ht}(\mathfrak{m}, s)+s-1 \geq t>s$, we shall assign any $N$ with $a_N \in \mathcal A(s)$, and make the assignment to be bijective. We now compute that
\begin{align*}
b_N'-a_{w(N)}'= & b_N+1-\frac{n-2N+1}{2}-(a_{w(N)}-\frac{n-2w(N)+1}{2}) \\
                = & b_N+1-\frac{n-2N+1}{2}-(b_N+1-\frac{n-2(N-s)+1}{2} \\
								=& s
\end{align*}
The second equality in the above equations follows from Lemma \ref{lem counting ht}. For the case that  $\mathrm{ht}(\mathfrak{m},s-x)+(s-x)-1\geq t >\mathrm{ht}(\mathfrak{m}, s-x+1)+(s-x)$, the proof is analogous to the case that $\mathrm{ht}(\mathfrak{m}, s)+s-1 \geq t>s$ with the use of Lemmas \ref{lem counting ht} and  \ref{lem bij two counting sets}. We omit the details.
\end{proof}

\subsection{$W$-representations in twisted-elliptic ladder representations}

Fix a choice $S$ of basic spin representation of $\widetilde{S}_l$ (as in Section \ref{ss dirac op}). 
\begin{lemma} \label{lem multiplicity spin and non}
Let $\mathfrak{m} \in \mathcal Z_{l}^{\mathrm{ladd}} \cap \mathcal Z_l^{\mathrm{ellip}}$. Let $\alpha =\alpha(\mathfrak{m})$. Let $\lambda=\lambda(\mathfrak{m}^{\mathrm{temp}})$. Then if $\alpha$ is not a hook or $l$ is odd,
\[  \mathrm{dim} \mathrm{Hom}_{\widetilde{W}}(\widetilde{\sigma}_{\lambda}, \sigma_{\alpha} \otimes S) \geq \frac{1}{\varepsilon_{\lambda} \varepsilon_{(l)}} 2^{(l(\lambda)-1)/2} .
\]
If $\alpha$ is a hook and $n$ is odd, then $ \mathrm{dim} \mathrm{Hom}_{\widetilde{W}}(\widetilde{\sigma}_{\lambda}, \sigma_{\alpha} \otimes S)$ is $0$ or $1$, depending on the choice of the basic spin representation $S$. 
\end{lemma}

\begin{proof}
Note that if $\alpha$ is a hook, then $\lambda=(n)$. The lemma follows from \cite[Theorem 9.3]{St}. In the notation of \cite[Theorem 9.3]{St}, we have $g_{ \lambda,\alpha} \geq 1$ since there is an obvious way to satisfy the criteria. We only give an example. We consider the partition $\alpha=(4,4,3,2,2)$ and the corresponding $\lambda=(8,6,1)$. Now we can fill the numbers $1,1',2,2',3,3'$ in the Young tableaux for $\alpha$ as \\
\[
\begin{Young}
1'&1 &1 &1\cr
1'&2'&2 & 2 \cr
1'&2' & 3  \cr
1'&2' \cr
1& 2 \cr
\end{Young}
\]
Then the word for $\alpha$ (terminology in \cite[p. 125]{St}) is $121'2'1'2'31'2'221'111$. Then we can work out the multiplicity $m_i(j)$ \cite[(8.4)]{St}. To check the lattice property \cite[Theorem 9.3(b)(1)]{St}, the interesting case is that $m_{1}(14)=m_2(14)=3$. In that case, $w_1=1 \neq 2, 2'$ and hence the lattice property is satisfied.
\end{proof}


\begin{theorem} \label{thm canonical unique}
Let $\mathfrak{m} \in \mathcal Z_l^{\mathrm{ladd}} \cap \mathcal Z_l^{\mathrm{ellip}}$. Let $\beta =\lambda(\mathfrak{m})^T$ and let $\alpha=\alpha(\mathfrak{m})$ (see Section \ref{ss partition alpha}). Then we have 
 \[ \mathrm{dim}~\mathrm{Hom}_W(\sigma_{\alpha},L(\mathfrak{m})|_W ) =\mathrm{dim}~\mathrm{Hom}_W(\sigma_{\beta}, L(\mathfrak{m})|_W)=1 .\]
\end{theorem}

\begin{proof}
We first consider (1). It is well-known that $\mathrm{dim}~\mathrm{Hom}_W(\sigma_{\beta}, L(\mathfrak{m})|_W)=1$ (see Corollary \ref{cor lowest type}). Let $p=\mathrm{dim}~\mathrm{Hom}_W(\sigma_{\alpha},L(\mathfrak{m})|_W ) $. By Proposition \ref{prop  existence can w}, we have $p \geq 1$. By \cite[Proposition 4.4.2]{BC} and a property of $D$ \cite[(6.2.3)]{BC3}, we have $H_D(L(\mathfrak{m})) =\mathrm{ker} D=\mathrm{ker} D^2$. We have that $D^2$ acts as a zero operator on the $\widetilde{\sigma}_{\lambda}$-isotypic component by the formula (\ref{eqn formul d2}). Thus with Lemma \ref{lem multiplicity spin and non}, we have 
\[ \mathrm{dim}~\mathrm{Hom}_{\widetilde{W}}(\widetilde{\sigma}_{\lambda}, H_D(X))=\mathrm{dim}~\mathrm{Hom}_{\widetilde{W}}(\widetilde{\sigma}_{\lambda}, \mathrm{ker} D^2) \geq \frac{1}{\varepsilon_{\lambda} \varepsilon_{(l)}} 2^{(l(\lambda)-1)/2}p
\]
Now Theorem \ref{thm dirac ladder} forces $p=1$. This proves (1).
\end{proof}


\section{Unequal parameter case for type $BC$} \label{s unequal bc}

Affine Hecke algebras for unequal parameters are constructed geometrically by Kato \cite{Ka0}. The representation theory of those algebras is further studied in \cite{CK} and \cite{CKK} and others. For the harmonic analysis approach, see the study of Slooten \cite{Sl} and Opdam-Solleveld \cite{OS}. 

\subsection{Graded Hecke algebra for type $C$} \label{ss gha bc}

Let $V_n=\mathbb{C}^n$. Let $\epsilon_1,\ldots, \epsilon_n$ be a standard basis for $V_n$ equipped with a bilinear form given by $\langle \epsilon_i, \epsilon_j \rangle= \delta_{ij}$. We also identify the dual space $V_n^{\vee}$ with $\mathbb{C}^n$ such that $f \in V_n^{\vee} \cong \mathbb{C}^n, v_2 \in V_n$, $f(v_2)=\langle f, v_2 \rangle$. Let 
\[R_n =\left\{ \epsilon_i-\epsilon_j : 1\leq i \neq j \leq n \right\} \cup \left\{ 2\epsilon_i : i=1,\ldots, n \right\} \]
 be a root system of type $C_n$. Let $\Pi_n=\left\{ \epsilon_i-\epsilon_{i+1}: i=1,\ldots, n-1 \right\} \cup \left\{ 2\epsilon_k \right\}$ be a fixed set of simple roots in $R_n$. Let $W_n$ be the Weyl group of type $C_n$.

For $m >0$, the graded Hecke algebra $\mathbb{H}_{n,m}$ (depending on $c$) of type $C_n$ is generated by the symbols $\left\{ t_w: w\in W \right\}$ and $\left\{ v \in V_n \right\}$ subject to the following relations:
\begin{enumerate}
\item $v_1v_2=v_2v_1$ for $v_1,v_2 \in V_n$;
\item $t_{w_1}t_{w_2}=t_{w_1w_2}$ for $w_1,w_2 \in W_n$;
\item $t_{s_{\alpha}}v-s_{\alpha}(v)t_{s_{\alpha}}=c_{\alpha}\alpha^{\vee}(v)$.
\item $c_{\alpha}=1$ if $\alpha$ is a short root and $c_{\alpha}=m$ if $\alpha$ is a long root.
\end{enumerate}

We drop the parameter $\mathbf r$ from our notation in Section \ref{ss gha} and it is not hard to recover statements with $\mathbf r$. Any $\mathbb{H}_{n,m}$-module can extend to a $\mathbb{Z}_2$-graded module since $\theta$ is the identity in this case.

We shall work on algebraic definitions:

\begin{definition} \label{def st mod unequal}
\begin{enumerate}
\item Let $J \subset \Pi$. Let $\mathbb{H}_J$ be the subalgebra of $\mathbb{H}$ generated by $t_{s_{\alpha}}$ ($\alpha \in J$) and $v \in V \subset S(V)$.
\item An $\mathbb{H}_{n,m}$-module is said to be $E$ a {\it standard module} if there exists an irreducible $\mathbb{H}_J$-module $U$ satisfying the relation that $E \cong \mathbb{H} \otimes_{\mathbb{H}_J}U$ and for any $S(V)$-weight $\gamma$, $\gamma(\omega_{\alpha})\geq 0$ for any $\alpha \in J$ and $\gamma(\alpha)<0$ for any $\alpha \in \Pi \setminus J$.
\item An $\mathbb{H}_{n,m}$-module $E$ is said to be tempered if $E$ is standard with $J=\Pi$ in (1).
\end{enumerate}
\end{definition}

By \cite[Section 3]{Lu3}, the above definition of standard modules coincide the one in Definition \ref{def standard mod} for geometric parameters.

Notice that the definition of temperedness differs from \cite{CK} by a sign. Results from \cite{CK} can be recovered to our setting by applying the Iwahori-Matsumoto involution in Section \ref{ss IM invol}. We shall implicitly use this.



\subsection{$\mathbb{A}_W$-module structure on tempered modules} \label{ss aw tempered}

We follow the arguments in the proof of \cite[Corollary 4.23]{CK} (also see the proof of \cite[Pg 1062, Claim C]{Ka}). For $k \geq 1$, let $m$ be in the open interval of $(k, k+1)$ with $k \geq 1$. Let $MP_m$ be the set of marked partitions (see \cite{Ka}, \cite[Definition 1.28]{CK} for the definition) parametrizing irreducible tempered modules of $\mathbb{H}_{n,m}$ with real central characters. We denote by $\mathrm{temp}_m(\tau)$ the tempered module parametrized by the marked partition $\tau \in MP_m$ at $m$. Let $m_0=\frac{2k+1}{2}$. For any $\tau \in MP_m=MP_{m_0}$, we have
\[  \mathrm{temp}_{m}(\tau) \cong \mathrm{temp}_{m_0}(\tau)
\]
as $W_n$-modules. The $W_n$-structures of $\mathrm{temp}_{m_0}(\tau)$ is then abstractly known from the generalized Springer correspondence (see the geometric parameters in \cite[Sec. 2.13(e)]{Lu1}). Let $\mathcal P_{n,m}$ be the set of all irreducible tempered modules with real central characters, equipped with the partial ordering $\leq_{n,m}$ from the corresponding generalized Springer correspondence. We then obtain the map 
\[  \Phi_{n,m} : \mathcal P_{n,m} \rightarrow \mathrm{Irr}W_n
\]
satisfying the upper triangular property (c.f. Theorem \ref{thm w struct spr}).

Let $\mathbb{A}_{W_n}=S(V_n) \rtimes W_n$ be the skew-group ring. For $\sigma \in \mathrm{Irr}W_n$, we can define an $\mathbb{A}_{W_n}$-module $\mathbf K^{n,m}_{\sigma}$ as the one in (\ref{eqn kato mod}) with respect to the ordering $\leq_{n,m}$.

Following the argument in Corollary \ref{cor struct h}, we have:

\begin{proposition} \label{prop unequal aw structure}
Let $E\in \mathcal P$. Let $\sigma=\Phi_{n,m}(E)$. Then 
\[ \overline{E}_{\sigma} \cong \mathbf{K}_{\sigma}^{n,m} .
\] 
Here $\overline{E}_{\sigma}$ is an $\mathbb{A}_W$-module defined similarly as the one in Definition \ref{def deform structure} (i.e. ignoring the term $\mathbf r$). 
\end{proposition}

\subsection{$\mathbb{A}_W$-module structure on standard modules}

We keep using notations in Section \ref{ss aw tempered}. We also need some information for the central characters of $\mathrm{temp}(\tau)$. In particular, let $m$ be in an open interval. Following the parametrization of the marked partition, for $\tau \in MP(m)$, the associated central character is given by the value $a$ and $s$ in the notation of \cite{CK} (also see \cite[Convention 1.32]{CK}). For $\tau_1,\tau_2 \in MP(m)$, $\mathrm{temp}_m(\tau_1)$ and $\mathrm{temp}_m(\tau_2)$ have the same central character if and only if $\mathrm{temp}_{m_0}(\tau_1)$ and  $\mathrm{temp}_{m_0}(\tau_2)$ have the same central character. Recall that $\mathrm{temp}_{m_0}(\tau_i)$ ($i=1,2$) comes from the generalized Springer correspondence, and we can write into the form $\mathrm{temp}_{m_0}(\tau_i)=E_{h_{e_i},{e_i},1,\zeta_i}$ for some nilpotent $e_i \in \mathcal N$ in the notation of Section \ref{ss deform struct rtm}. Since $h_{e_i}$ determines the central character of $E_{h_{e_i},{e_i},1,\zeta_i}$, we must have that $e_1$ and $e_2$ lie in the same nilpotent orbit. 

Let $E$ be a standard module with a real central character. Then there exists $J\subset \Pi$ such that $E=\mathbb{H} \otimes_{\mathbb{H}_J}U$ with some $\mathbb{H}_J$-module $U$ described in Definition \ref{def st mod unequal}. Let
\[  V_J^{\bot}= \left\{ v \in V_n: \langle v,\alpha \rangle =0 \quad \mbox{ for any $\alpha \in J$ } \right\}.
\]
Now let $U'$ be an $\mathbb{H}_J$-module isomorphic to $U$ as vector space and the $\mathbb{H}_J$-module structure is determined by:
\[ \pi_{U'}(t_w)u=\pi_U(t_w)u \quad (w \in W), \]
\[ \pi_{U'}(\alpha) u =\pi_U(\alpha)u \quad (\alpha \in J),\]
\[ \pi_{U'}(v)u =0 ,\]
where $\pi_{U'}$ (resp. $\pi_U$) is the map defining the action of $\mathbb{H}_J$ on $U'$ (resp. $U$).

Now set $E_T=\mathbb{H}\otimes_{\mathbb{H}_J}U'$. One checks from Definition \ref{def st mod unequal}(2) that $E_T$ is tempered. We now let $E_1,\ldots, E_r$ be the collection of all the irreducible composition factors for $E_T$. By the weight considerations, we deduce that $E_1,\ldots, E_r$ have the same central character. (Hence they have to correspond to the same nilpotent orbit in the sense of the first paragraph of this subsection.) Write the central character for $E_i$ as $W_nh$ and we shall call $\langle h, h\rangle$ to be the length of $E_T$. 

Now with the weight considerations as in the Langlands classification (see \cite{Ev, KR}) and suitable ordering using the generalized Springer correspondence, we can deduce a version of Corollary \ref{cor lowest type} in this setting. From there, we can deduce a version of Lemma \ref{lem filt kato mod}:

\begin{lemma} \label{lem filt kato mod uneq}
Let $E$ be a standard module of $\mathbb{H}_{n,m}$. Let $E_T, E_1,\ldots, E_r$ be tempered modules as above. Let $\sigma_i$ be the lowest type for $E_i$ for $i=1,\ldots, r$. Set $\sigma=\sigma_1$. Then there exists a $\mathbb{A}_{W_n}$-module filtration $\left\{ \overline{Y}_i \right\}_{i=0}^p$ on $\overline{E}_{\sigma}$ such that 
\[ 0=\overline{Y}_0 \subset \overline{Y}_1 \subset \overline{Y}_2 \subset \ldots \subset \overline{Y}_k=\overline{E}_{\sigma} \]
and for $i=0,\ldots, p-1$, $\overline{Y}_{i+1}/\overline{Y}_i$ is isomorphic to a module $\mathbf{K}^{n,m}_{\sigma_j}$ for some $\sigma_j$. 
\end{lemma}

\subsection{Dirac cohomology of tempered modules}

This section is essentially in the work of \cite{Ci}. We provide some explanations since the statement we look for  is not stated in \cite{Ci}.

Let $k\geq 1$. Following Section 3.1 in \cite{CK}, we consider $m \notin \frac{1}{2}\mathbb{Z}$ and $m \in (k,k+1)$. We firstly consider a partition $\lambda$ of $n$. According to \cite[Theorem 3.5]{CK} (also see \cite[Theorem 3.16]{CK}), $\lambda$ defines a tempered module, denoted $E^{\lambda}$, with the central character $Ws_{\lambda}$. Here $s_{\lambda} \in \mathbb{C}^n$ is given as in \cite[Section 3.1]{CK}, and can be computed from \cite[Definition 1.28]{CK}.

\begin{example}
We consider $\mathbb{H}_{7,m}$. Identify $\mathbb{C}^7$ with $\mathfrak{h}$ with $\langle, \rangle$ and the standard basis. Let $\sigma=(3,2,2)$. Then $W_7s_{\lambda}= W_7(m,m+1,m+2,m-1,m, m-2,m-1)$.
\end{example}

We now study the Dirac cohomology of $E^{\lambda}$. One can apply the formulas in \cite{Ci}. For a partition $\lambda_m$ of $n$, we view it as a Young tableaux. For the box in the $i$-th row and $j$-th column, we define $c(i,j)=m+i-j$, called a $c$-content. For a partition $\lambda$ of $n$, we define $p_l(\lambda+m)$ to be the sum of $l$-th power of all the $c$-contents in the boxes of (the Young tableaux of) $\lambda$. We have $p_2(\lambda_m)=p_2(\lambda_0)+2mp_1(\lambda_0)+nm^2$. Using \cite[Section 3.7]{Ci}, we obtain a version of \cite[Theorem 1.1]{Ci}. With \cite{BCT} and the $W$-structure of the generalized Springer correspondence in the previous subsection, we have:

\begin{proposition} \label{prop iso class nonzero dirac}
Let $k\geq 1$. Let  $m \in (k,k+1)$ with $m \neq \frac{1}{2}\mathbb{Z}$. Let $P_n$ be the cardinality of the set of partition of $n$. There is exactly $P_n$ isomorphism classes of tempered modules such that $H_D(E)\neq 0$. The central characters of those tempered modules are given by $Ws_{\lambda}$ described above, where $\lambda$ runs for all partitions of $n$.
\end{proposition}

\subsection{Dirac cohomology of standard modules}

\begin{theorem}
Let $m\geq 1$. Let $\mathbb{H}_{n,m}$ be defined in Section \ref{ss gha bc}. Let $E$ be a standard module of $\mathbb{H}_{n,m}$ (Definition \ref{def st mod unequal}). Then $H_D(E) = 0$ if and only if $E$ is not a tempered module in Proposition \ref{prop iso class nonzero dirac}.
\end{theorem}

\begin{proof}
When $m \geq 1$ and $m \in \frac{1}{2}\mathbb{Z}$, $\mathbb{H}_{n,m}$ is isomorphic to a geometric graded Hecke algebra of Lusztig \cite{Lu1}. That case has been dealt in Theorem \ref{thm vanihsing dirac coh}. We only have to consider $m \notin \frac{1}{2}\mathbb Z$. With Lemma \ref{lem filt kato mod uneq}, Propositions \ref{prop unequal aw structure} and \ref{prop iso class nonzero dirac}, we can modify the arguments in Theorem \ref{thm vanihsing dirac coh} to obtain the statement. We omit the details.
\end{proof}

\subsection{A remark on noncrystallographic types}

Other than unequal parameter case, a possible generalization is for $R$ of noncrystallographic types. Results in Section \ref{s def dirac coh} still apply. We also have some other partial information in the literature. Some structural information for tempered modules can be found in \cite{KR} and \cite{Kr2}. For the Dirac cohomology, some computations are done in \cite{Ch1}.

\end{document}